\documentclass[11pt,a4paper]{article}
\usepackage{graphicx}
\usepackage{mathrsfs}
\usepackage{color}
\usepackage[latin1]{inputenc}
\usepackage[francais, english]{babel}

\usepackage{libertine}
\usepackage[T1]{fontenc}
\usepackage{lipsum}
\usepackage{titlesec}
\usepackage[colorlinks=true,linkcolor=magenta,citecolor=blue]{hyperref}

\usepackage[cmex10]{amsmath}
\usepackage{stmaryrd}
\usepackage{color}
\usepackage{cmll}
\usepackage{amsmath,amssymb}
\usepackage[sloped]{fourier}
\usepackage{amsthm}
\usepackage{esint}
\usepackage{psfrag}

\author{S.Alvarez, N.Hussenot}
\date{}
\title{Singularities for analytic continuations of holonomy germs of Riccati foliations\\
       (Singularités des extensions analytiques des germes d'holonomie des feuilletages de Riccati)}

\begin{document}

\newtheorem{maintheorem}{Theorem}
\newtheorem{maincoro}[maintheorem]{Corollary}
\renewcommand{\themaintheorem}{\Alph{maintheorem}}
\newcounter{theorem}[section]
\newtheorem{exemple}{\bf Exemple \rm}
\newtheorem{exercice}{\bf Exercice \rm}
\newtheorem{conj}[theorem]{\bf Conjecture}
\newtheorem{defi}[theorem]{\bf Definition}
\newtheorem{lemma}[theorem]{\bf Lemma}
\newtheorem{proposition}[theorem]{\bf Proposition}
\newtheorem{coro}[theorem]{\bf Corollary}
\newtheorem{theorem}[theorem]{\bf Theorem}
\newtheorem{rem}[theorem]{\bf Remark}
\newtheorem{ques}[theorem]{\bf Question}
\newtheorem{propr}[theorem]{\bf Property}
\newtheorem{question}{\bf Question}
\def\bp{\noindent{\it Proof. }}
\def\ep{\noindent{\hfill $\fbox{\,}$}\medskip\newline}
\renewcommand{\theequation}{\arabic{section}.\arabic{equation}}
\renewcommand{\thetheorem}{\arabic{section}.\arabic{theorem}}
\newcommand{\eps}{\varepsilon}
\newcommand{\disp}[1]{\displaystyle{\mathstrut#1}}
\newcommand{\fra}[2]{\displaystyle\frac{\mathstrut#1}{\mathstrut#2}}
\newcommand{\dif}{{\rm Diff}}
\newcommand{\homeo}{{\rm Homeo}}
\newcommand{\Per}{{\rm Per}}
\newcommand{\Fix}{{\rm Fix}}
\newcommand{\A}{\mathcal A}
\newcommand{\Z}{\mathbb Z}
\newcommand{\Q}{\mathbb Q}
\newcommand{\R}{\mathbb R}
\newcommand{\C}{\mathbb C}
\newcommand{\N}{\mathbb N}
\newcommand{\T}{\mathbb T}
\newcommand{\U}{\mathbb U}
\newcommand{\D}{\mathbb D}
\newcommand{\PP}{\mathbb P}
\newcommand{\Sp}{\mathbb S}
\newcommand{\K}{\mathbb K}
\newcommand{\car}{\mathbf 1}
\newcommand{\g}{\mathfrak g}
\newcommand{\gs}{\mathfrak s}
\newcommand{\h}{\mathfrak h}
\newcommand{\rr}{\mathfrak r}
\newcommand{\fhi}{\varphi}
\newcommand{\ffhi}{\tilde{\varphi}}
\newcommand{\moins}{\setminus}
\newcommand{\ds}{\subset}
\newcommand{\W}{\mathcal W}
\newcommand{\WW}{\widetilde{W}}
\newcommand{\F}{\mathcal F}
\newcommand{\G}{\mathcal G}
\newcommand{\CC}{\mathcal C}
\newcommand{\RR}{\mathcal R}
\newcommand{\DD}{\mathcal D}
\newcommand{\M}{\mathcal M}
\newcommand{\B}{\mathcal B}
\newcommand{\cS}{\mathcal S}
\newcommand{\HH}{\mathcal H}
\newcommand{\Hyp}{\mathbb H}
\newcommand{\UU}{\mathcal U}
\newcommand{\Pp}{\mathcal P}
\newcommand{\QQ}{\mathcal Q}
\newcommand{\E}{\mathcal E}
\newcommand{\GG}{\Gamma}
\newcommand{\LL}{\mathcal L}
\newcommand{\KK}{\mathcal K}
\newcommand{\TT}{\mathcal T}
\newcommand{\X}{\mathcal X}
\newcommand{\Y}{\mathcal Y}
\newcommand{\ZZ}{\mathcal Z}
\newcommand{\bE}{\overline{E}}
\newcommand{\bF}{\overline{F}}
\newcommand{\wF}{\widetilde{F}}
\newcommand{\hcF}{\widehat{\mathcal F}}
\newcommand{\bW}{\overline{W}}
\newcommand{\bcW}{\overline{\mathcal W}}
\newcommand{\tL}{\widetilde{L}}
\newcommand{\diam}{{\rm diam}}
\newcommand{\diag}{{\rm diag}}
\newcommand{\Jac}{{\rm Jac}}
\newcommand{\Plong}{{\rm Plong}}
\newcommand{\Tr}{{\rm Tr}}
\newcommand{\Conv}{{\rm Conv}}
\newcommand{\Ext}{{\rm Ext}}
\newcommand{\Spec}{{\rm Sp}}
\newcommand{\Isom}{{\rm Isom}\,}
\newcommand{\Supp}{{\rm Supp}\,}
\newcommand{\Grass}{{\rm Grass}}
\newcommand{\Hold}{{\rm H\ddot{o}ld}}
\newcommand{\Ad}{{\rm Ad}}
\newcommand{\ad}{{\rm ad}}
\newcommand{\e}{{\rm e}}
\newcommand{\s}{{\rm s}}
\newcommand{\pol}{{\rm pole}}
\newcommand{\Aut}{{\rm Aut}}
\newcommand{\End}{{\rm End}}
\newcommand{\Leb}{{\rm Leb}}
\newcommand{\Liouv}{{\rm Liouv}}
\newcommand{\Lip}{{\rm Lip}}
\newcommand{\Int}{{\rm Int}}
\newcommand{\cc}{{\rm cc}}
\newcommand{\grad}{{\rm grad}}
\newcommand{\proj}{{\rm proj}}
\newcommand{\mass}{{\rm mass}}
\newcommand{\dive}{{\rm div}}
\newcommand{\dist}{{\rm dist}}
\newcommand{\im}{{\rm Im}}
\newcommand{\re}{{\rm Re}}
\newcommand{\codim}{{\rm codim}}
\newcommand{\Map}{\longmapsto}
\newcommand{\vide}{\emptyset}
\newcommand{\tr}{\pitchfork}
\newcommand{\ssl}{\mathfrak{sl}}

\newenvironment{demo}{\noindent{\textbf{Proof.}}}{\quad \hfill $\square$}
\newenvironment{pdemo}{\noindent{\textbf{Proof of the proposition.}}}{\quad \hfill $\square$}
\newenvironment{IDdemo}{\noindent{\textbf{Idea of proof.}}}{\quad \hfill $\square$}

\def\to{\mathop{\rightarrow}}
\def\act{\mathop{\curvearrowright}}
\def\To{\mathop{\longrightarrow}}
\def\Sup{\mathop{\rm Sup}}
\def\Max{\mathop{\rm Max}}
\def\Inf{\mathop{\rm Inf}}
\def\Min{\mathop{\rm Min}}
\def\lims{\mathop{\overline{\rm lim}}}
\def\limi{\mathop{\underline{\rm lim}}}
\def\egal{\mathop{=}}
\def\dans{\mathop{\subset}}
\def\surj{\mathop{\twoheadrightarrow}}

\def\h#1#2#3{ {\rm{hol}}^{#1}_{#2\rightarrow#3}}

\maketitle

\begin{abstract}
In this paper we study the problem of analytic extension of holonomy germs of algebraic foliations. More precisely we prove that for a Riccati foliation associated to a branched projective structure over a finite type surface which is non-elementary and parabolic, all the holonomy germs between a fiber and the corresponding holomorphic section of the bundle are led to singularities by almost every developed geodesic ray. We study in detail the distribution of these singularities and prove in particular that they form a dense uncountable subset of the limit set. This gives another negative answer to a conjecture of Loray (already disproved in \cite{CDFG}) using a completely different method, namely the ergodic study of the foliated geodesic flow initiated in \cite{BG,BGV,BGVil}. 
\end{abstract}

\selectlanguage{francais}
\begin{abstract}
Dans cet article, nous étudions le problème d'extension analytique de germes d'holonomie de feuilletages algébriques. Plus précisément, nous démontrons que pour un feuilletage de Riccati associée à une structure projective branchée sur une surface de type fini qui est non-élémentaire et parabolique, tous les germes d'holonomies entre une fibre et la section holomorphe du fibré vertical correspondante sont conduits vers une singularité par presque tout chemin géodésique développé. Nous étudions en détail la distribution de ces singularités et prouvons en particulier qu'elles forment une partie dense et indénombrable de l'ensemble limite. Cela redonne une réponse négative à une conjecture de Loray (qui avait déjà été infirmée dans \cite{CDFG}) en utilisant une méthode complètement différente: l'étude ergodique du flot géodésique feuilleté initiée dans \cite{BG,BGV,BGVil}. 
\end{abstract}

\selectlanguage{english}

\section{Introduction}

\paragraph{Analytic continuation of holonomy maps.} The present paper is devoted to the problem of extending analytically holonomy germs of holomorphic foliations of the complex projective plane, and of algebraic surfaces such as $\C\PP^1$-bundles over compact Riemann surfaces. The study of \emph{holonomy maps}, or \emph{Poincaré maps}, of a foliation is of special interest since they encode the dynamical behaviour of its leaves. For example a fixed point of a holonomy map corresponds to a periodic leaf of the foliation: analytic properties of holonomy maps are closely related to various interesting and famously difficult questions concerning periodic leaves such as their number or their persistence.

For example in \cite{FRY} Françoise, Roytvarf and Yomdin study the analytic continuations, the fixed points and the singularities of holonomy maps of the Abel differential equation in relation with Pugh's problem about the number of isolated real periodic solutions. 

In \cite{I2} Ilyashenko relates the problems of simultaneous uniformization of the leaves of a foliation of $\C^k$ by analytic curves (with a uniformizing function which depends analytically on the initial condition) and of persistence of complex limit cycles. The relation he finds is closely related to the extension property: the non-extendability of holonomy maps is an obstruction to simultaneous uniformization. This led Ilyashenko to ask whether holonomy germs of generic polynomial vector fields exhibit algebraic or transcendental behaviour, namely if they can be analytically continued along \emph{most} real rays (see Problems 8.6. and 8.7. of \cite{I1} and Problem 8 of \cite{I3}).

\paragraph{Loray's conjecture.} In his study of Painlevé's work on algebraic differential equations Loray states the following conjecture (Conjecture 1 of \cite{Lo}).

\emph{A germ of holonomy map} $h:(T_0,p_0)\to(T_1,p_1)$ \emph{between two algebraic transversals of an algebraic foliation of $\C\PP^2$ can be analytically continued along every path which avoids a countable number of points of} $T_0$ \emph{called the singularities.}

His idea is that if this conjecture were established it would be possible to replace the study of the \emph{holonomy pseudogroup} by that of a (possibly very consistent) group and that a Galois theory for algebraic foliations similar to the one described in \cite{K} could be derived from it.

In \cite{CDFG} Calsamiglia, Deroin, Frankel and Guillot give a very precise answer to this conjecture. Their answer depends on the dynamical properties of the leaves. If a foliation has rich contracting dynamics, then Loray's conjecture does not hold for this foliation. In the other case it holds true. More precisely, they prove the following:
\begin{enumerate}
\item \emph{Loray's conjecture holds true for singular foliations given by closed meromorphic} $1$-\emph{forms on the complex projective plane}.
\item \emph{A Riccati foliation whose holonomy representation is given by uniformization has holonomy germs between lines with a natural boundary}.
\item \emph{A Riccati foliation whose holonomy representation is parabolic with a dense image in} $PSL_2(\C)$ \emph{has holonomy germs between lines with a full singular set}.
\item \emph{And finally Loray's conjecture is false for a generic foliation of} $\C\PP^2$: \emph{such a foliation possesses a holonomy germ from a line to an algebraic curve whose singular set contains a Cantor set}.
\end{enumerate}

Even if the fourth property is the most spectacular, in this paper we will focus on the second and third ones. The proof of the third property consists in using the density of the holonomy group in $PSL_2(\C)$ in order to construct inductively paths which lead to singularities. A question arose naturally: \emph{are the holonomy germs led to singularities by a ``generic'' path}?

Of course the term \emph{generic} has to be precised. In \cite{H} the second author studied the analytic continuation of holonomy germs along Brownian paths and the answer he found was quite unexpected (see Theorem \ref{hussenot1} for the precise statement).

\emph{Given a Riccati foliation whose holonomy group is parabolic and acts minimally on} $\C\PP^1$, \emph{any holonomy germ between two lines} $T_1$ \emph{and} $T_2$ \emph{can be analytically continued along almost all Brownian paths}.

The context of Riccati foliation is not generic, but proves to be an excellent one when we want to explore the links between dynamics of foliations and extendability properties of holonomy maps. Indeed \cite{CDFG,H} both use the ``duality'' between holonomies between lines of Riccati foliations and projective structures on surfaces of finite type that we shall describe below.

\paragraph{Riccati foliations and projective structures.} Riccati equations are of the form:
\begin{equation}
\label{Eq:Riccati}
\frac{dy}{dx}=a(x)y^2+b(x)y+c(x),
\end{equation}
where $a,b,c$ are rational functions of the complex variable $x$. It is well known (see \cite{Hi}) that this equation is in reality a disguised linear differential equation $d\mathbf{w}/dx=A(x)\mathbf{w}$ where $\mathbf{w}\in\C^2$ and $A(x)$ is a $2\times 2$ matrix whose entries depend rationally on $x$, and that they are characterized by the possibility of finding locally a basis of local solutions which can be analytically continued along every path avoiding the finite set of poles of $A$ that we denote by $(A)_{\infty}$. This gives rise to a holonomy representation $\widetilde{\rho}:\pi_1(\Sigma)\to GL_2(\C)$, where $\Sigma=\C\PP^1\moins(A)_{\infty}$.

Hence the \emph{Riccati foliation} of $\C\PP^1\times\C\PP^1$ given by the analytic continuations of the solutions of this equation has the following description. There are only a finite number of invariant fibers (vertical leaves with singularities) which are the $L_{x_0}=\{x_0\}\times\C\PP^1$, where $x_0\in(A)_{\infty}$, and any other leaf is \emph{everywhere} transverse to the vertical fibers. Moreover the $\C\PP^1$-bundle over $\Sigma$ that we denote by $\Pi:M\to\Sigma$ obtained by removing the invariant fibers is exactly the one obtained by suspension of the projectivization of $\widetilde{\rho}$ denoted by $\rho:\pi_1(\Sigma)\to PSL_2(\C)$: we denote by $\F$ the induced Riccati foliation on $M$.

Let $\overline{S}$ be a line in $\C\PP^1\times\C\PP^1$ which is not vertical (or any holomorphic section of the vertical bundle). Then $S=\overline{S}\moins\bigcup_{x_0\in (A)_{\infty}} L_{x_0}$ is a holomorphic copy of $\Sigma$ which is everywhere transverse to $\F$ except maybe at a finite number of tangency points. The holonomy of $\F$ between $S$ and any vertical fiber $F_p\simeq\C\PP^1$ then defines naturally a \emph{branched projective structure} on $\Sigma$ which is determined by a development-holonomy pair $(\DD,\rho)$ where $\rho$ is exactly the holonomy representation of the equation and $\DD$ is a nonconstant holomorphic map from $\Hyp$, the universal cover of $\Sigma$, to $\C\PP^1$ whose critical points are the lifts of the tangency points and which is $\rho$-equivariant.

A similar construction can be performed when the variable $x$ describes a more general algebraic curve. This leads us to consider Riccati foliations obtained by suspension of representations $\rho:\pi_1(\Sigma)\to PSL_2(\C)$, as well as branched projective structures on more general hyperbolic surfaces of finite type $\Sigma$.

\paragraph{Main result.} We will study the problem of analytic extension of germs of holonomy maps along ``generic paths''. In \cite{H} ``generic'' meant  typical for the \emph{Brownian motion}. Here it will mean typical for the \emph{geodesic flow}. And the result we obtain is the exact opposite answer.

Recall that a path leads a germ of holomorphic map between Riemann surfaces to a \emph{singularity} if analytic continuation can be performed along the path, but not beyond its extremity.

\begin{maintheorem}
\label{analyticccontinuation}
Let $\Sigma$ be a hyperbolic surface of finite type. Consider a branched projective structure on $\Sigma$, represented by a development-holonomy pair $(\DD,\rho)$, that is parabolic and non-elementary. Let $(\Pi,M,\Sigma,\C\PP^1,\F)$ be the associated Riccati foliation, $\sigma^0$ be the holomorphic section of $\Pi$ associated to $\DD$, and $S=\sigma^0(\Sigma)$.

Let $p\in\Sigma$ and $x\in F_p$. Then every holonomy germ $h$ at $x$ between $F_p$ and $S$ is led to a singularity along a typical developed geodesic ray starting at $x$.

Moreover this set of singularities is an uncountable dense subset of the limit set $\Lambda_{\rho}$. Better: it is distributed according to the harmonic measure $m_p$.
\end{maintheorem}

Let us explain some of the terms appearing in the statement. A branched projective structure $(\DD,\rho)$ is said to be \emph{parabolic} if the developing map reads as $z\mapsto\log z/(2\mathbf{i}\pi)$ in a holomorphic coordinate $z$ around each puncture. This implies that the holonomy map around a puncture is conjugated to a translation.

We call \emph{developed geodesic rays} the images of geodesic rays by the developing map, and we say that a developed geodesic ray starting at a point $x$ is typical if it is the image of a ray which is typical for the Lebesgue measure.

When the representation does not preserve a measure on $\C\PP^1$, we say that the structure is \emph{non-elementary}. Its action on the sphere has a unique minimal set $\Lambda_{\rho}$ that we call the \emph{limit set} of $\rho$. Moreover \cite{DD} introduced the notion of family of \emph{harmonic measures} which is the unique family of probability measures $(m_z)_{z\in\Hyp}$ on $\C\PP^1$ which is $\rho$-equivariant and \emph{harmonic}, meaning that for every Borel set $A\dans\C\PP^1$ the map $z\mapsto m_z(A)$ is harmonic for the Laplace operator.

In particular all these measures are equivalent (by the mean property) and supported by the limit set $\Lambda_{\rho}$ (by equivariance). By equivariance this gives a family $(m_p)_{p\in\Sigma}$ on the fibers $F_p$ of the foliated bundle.

\paragraph{Compactification of Riccati foliations.} Theorem \ref{analyticccontinuation} deals with analytic continuation of holonomy germs of a \emph{non-singular} foliation transverse to a $\C\PP^1$-bundle over a \emph{non-compact} hyperbolic surface $\Sigma$.

When the foliation is associated to a \emph{parabolic} branched projective structure, it can be compactified (see Section 3.2 of \cite{CDFG}, and Section 1.2 of \cite{DD}). This is done by gluing over the cusps a local model for \emph{meromorphic} flat connection on the disc $\D$, with a single pole at $0$ and parabolic monodromy. Different models of this sort can be found in Brunella's work on the birational geometry of foliations \cite{Br}.

This way, we obtain a $\C\PP^1$-bundle over a compact surface, whose fibers are transverse to a \emph{singular} foliation, except a finite number of them. The section $S$ also compactifies as a complex curve $\overline{S}$, and the conclusion of Theorem \ref{analyticccontinuation} also holds for holonomy germs from a generic fiber to $\widetilde{S}$. In that sense, our result disproves Loray's conjecture for algebraic foliations of $\C\PP^1$-bundles over compact surfaces.

\paragraph{The dual result.} In order to prove Theorem \ref{analyticccontinuation}, we first give the answer to a dual problem, which can be expressed in terms of branched projective structures.

There is a duality between holonomy maps and development which is used in \cite{CDFG,H}. Consider a Riccati foliation $\F$ and a local holonomy map $h$ between a fiber $F_p$, $p\in\Sigma$ and a holomorphic section of $\Sigma$. By definition $h$ is the inverse of the developing map $\DD$ restricted to the range of $h$. We shall focus on the research of \emph{asymptotic singularities} of continuations of $h$, i.e. limits of paths $\DD(c)$ where $c:[0;\infty)\to\Hyp$ has a limit in $\partial\Hyp=\R\PP^1$. Hence there are two problems which are dual.
\begin{enumerate}
\item Prove that for a path $c:[0;\infty)\to\Hyp$ having a limit in $\R\PP^1$, the path $\gamma=\DD(c)$ has a limit in $\C\PP^1$.
\item Prove that the path $\gamma$ leads the holonomy germ $h:(F_p,x)\to(S,x')$ to an asymptotic singularity, where $x$ denotes $\gamma(0)$ and $x'$ denotes the projection on $S$ of $c(0)$.
\end{enumerate}

Let us introduce the main dynamical character of this paper. The leaves of a Riccati foliation $\F$ are naturally endowed with a hyperbolic metric (by lifting the metric of the base) so that it is possible to consider the \emph{foliated geodesic flow} $G_t$ on the unit tangent bundle of the foliation $T^1\F$. This flow possesses a weak form of hyperbolicity called in \cite{BGM} \emph{foliated hyperbolicity} (see also the first author's thesis \cite{Al4}). We will use the ergodic properties of this flow in order to prove our main results: our method is very much in the spirit of \cite{Al3,BG,BGV,BGVil}.

Since the leaves are locally isometric to the base, the foliated geodesic flow projects down to the geodesic flow of the base and sends fibers to fibers as a projective map (see Paragraph \ref{foliatedflows}): it is a \emph{projective cocycle}. Under the condition of parabolicity of the structure, \cite{BGVil} proves that Oseledets' theorem applies and that Lyapunov exponents exist. If moreover the holonomy representation does not preserve a probability measure on $\C\PP^1$, then a combination of the works of Avila-Viana \cite{AV} and Ledrappier-Sarig \cite{LS} yields the positivity of the top Lyapunov exponent (see Theorem \ref{positivitylyapexp}). Oseledets' theorem then provides two measurable \emph{Lyapunov sections} $\sigma^-,\sigma^+:T^1\Sigma\to T^1\F$ well defined on a Borel set full for the Liouville measure. Consider the lifts $\widetilde{\sigma}^-,\widetilde{\sigma}^+:T^1\Hyp\to T^1\Hyp\times\C\PP^1$. The dual result of Theorem \ref{analyticccontinuation} is:

\begin{maintheorem}
\label{limitedprojective}
Let $\Sigma$ be a hyperbolic surface of finite type. Consider a branched projective structure on $\Sigma$, represented by a development-holonomy pair $(\DD,\rho)$, that is parabolic and non-elementary. Then for every $z\in\Hyp$ and $d\theta$-almost every $v\in T^1_z\Hyp$:
$$\DD(c_v(t))\To_{t\to\infty}\widetilde{\sigma}^-(v),$$
where $c_v$ represents the geodesic ray directed by $v$, and $\widetilde{\sigma}^-$ is the lift of the Lyapunov section $\sigma^-$ to $T^1\Hyp$.
\end{maintheorem}

\paragraph{Remark.} It is interesting to note that the set of limits of typical developed geodesic rays described in the theorem above is independent of the choice of a developing map. It only depends on the holonomy of the underlying Riccati foliation, and is defined by looking a nonuniformly hyperbolic attractor of the foliated geodesic flow. A more detailled description of this set is given by the following theorem.

\paragraph{Distribution of the singularities.} Theorem \ref{analyticccontinuation} shows that although in hyperbolic geometry almost every Brownian path possesses a geodesic escort, there exists a qualitative difference between the geodesic flow and the Brownian motion which is due to the fluctuations of the latter. The remarkable fact is that at the ergodic level, we don't see the difference: almost every Brownian path spends almost all of its time close to the limit of its developed geodesic escort. More precisely, in \cite{H} it is proven under the hypothesis of Theorem \ref{analyticccontinuation} (see Theorem \ref{hussenot2}) that for almost every Brownian path $\omega$ on $\Hyp$, there exists $\e(\omega)$ such that:
$$\lim\limits_{t\to\infty}\frac{1}{t}\int_0^t\DD*\delta_{\omega(s)}ds=\delta_{\e(\omega)}.$$

We prove that these points $\e(\omega)$ and the limits of developed geodesic rays are distributed according to the same law. The main goal of Section \ref{distributionsingularities} is to study this distribution in detail. As a corollary of the propositions proven in that section we get the following result:

\begin{maintheorem}
\label{distriblimit}
Let $\Sigma$ be a hyperbolic surface of finite type. Consider a branched projective structure on $\Sigma$, represented by a development-holonomy pair $(\DD,\rho)$, that is parabolic and non-elementary. Let $(\Pi,M,\Sigma,\C\PP^1,\F)$ be the associated Riccati foliation.

Denote by $(s_p)_{p\in \Sigma}$ the family of limits of distributions of developed geodesic rays. Then this family coincides with:
\begin{itemize}
\item the family of conditional measures of the projection of the unique SRB measure (in the sense of \cite{BR,Si}) for the foliated geodesic flow via the canonical map $pr:T^1\F\to M$;
\item the family of conditional measures of the unique foliated harmonic measure (in the sense of \cite{Gar}) for $\F$;
\item the unique family of $\nu_p$-stationary measures where $(\nu_p)_{p\in\Sigma}$ is a family of probability measures on $\pi_1(\Sigma)$ obtained by the procedure of Furstenberg-Lyons-Sullivan's discretization of the Brownian motion;
\item the family of distributions of points $\e(\omega)$, $\omega$ Brownian path starting at $p$;
\item the family of harmonic measures of $\rho$.
\end{itemize}
\end{maintheorem}

This theorem contains implicit statements, namely the uniqueness of the SRB measure and of the foliated harmonic measure. It is the occasion to review in a unified way previous results of \cite{Al1,Al2,DD,H,Ma}: we carefully explain the link between each of these measures in Section \ref{distributionsingularities}.

\paragraph{Organization of the paper.} In Section \ref{prelim}, we give the main definitions and results which will be used throughout this paper. In particular we give a discussion about Lyapunov exponents of the cocycle defined by the foliated geodesic flow in this noncompact setting. In Section \ref{distributionsingularities} we analyze the distribution of the limit of developed geodesic rays and prove Theorem \ref{distriblimit}. We also give a proof that this set of limit points is uncountable and dense in the limit set. In Section \ref{limitsgeodesicrays} we show how to deduce Theorem \ref{limitedprojective} from ergodic-theoretical facts as well as from an integrability result. Section \ref{Proofofintegrability} is the main technical section: we prove the aforementioned integrability result.

\paragraph{Notations.}
In all what follows, we will use the following notations:
\begin{itemize}
\item $\dist_{\C\PP^1}$ which stands for the Fubini-Study distance in $\C\PP^1$;
\item $\dist_{\Hyp}$ which stands for the hyperbolic distance in $\Hyp$;
\item $\dist_{\C}$ which stands for the euclidian distance in $\C$.
\end{itemize}

\section{Preliminaries}
\label{prelim}

\paragraph{Analytic continuation.} Let $X_0,X_1$ be two Riemann surfaces, and $f:(X_0,x_0)\to(X_1,x_1)$ be a germ of holomorphic map. We say that $f$ admits an \emph{analytic continuation along a path} $c:[0;1]\to X_0$ if there exists a chain of discs $D_0,...,D_n$ which cover $c$, as well as a sequence of holomorphic maps $f_k:D_k\to X_1$, such that the germ of $f_0$ at $x_0$ is given by $f$, and $f_k=f_{k+1}$ in restriction to $D_k\cap D_{k+1}$. The germ of $f_n$ at $c(1)$ is called the \emph{determination} of $f$ over $c(1)$ and depends only on the homotopy class of $c$ \emph{inside the holomorphy domain of the germ}.

\paragraph{Singularities.} We say that \emph{a path} $c:[0;1]\to X_0$ \emph{leads the germ} $f$ \emph{to a singularity} if $f$ can be extended analytically along each path $c_{|[0;1-\eps]}$, but not along $c$. The point $c(1)$ will be called \emph{singularity} of $f$.

\subsection{Hyperbolic surfaces of finite type and their projective structures}

\paragraph{Hyperbolic surfaces of finite type.} In the sequel, we shall consider hyperbolic Riemannian surfaces $\Sigma$ which are \emph{not compact and with finite area}: such a surface will be said to be \emph{of finite type}. By definition, they are uniformized by the upper half plane $\Hyp=\{z\in\C;\,\im(z)>0\}$ endowed with the \emph{Poincaré metric}:
$$ds^2=\frac{dx^2+dy^2}{y^2}.$$

A hyperbolic surface of finite type $\Sigma$ is biholomorphic to $\Sigma_g\moins \{p_1,...,p_k\}$ where $\Sigma_g$ is a compact Riemann surface of genus $g$. Neighbourhoods of the $p_i$ are called \emph{cusps}. 

It is well known that in this case the fundamental group of $\Sigma$ is a free group, and that there is a fundamental domain $P$ for a copy $\pi_1(\Sigma)\simeq\Gamma<PSL_2(\R)$, which is an ideal polygon with $2l$ vertices at infinity, where $l$ is the maximal number of mutually disjoint non-homotopic geodesics whose ends arrive to punctures.

Such a surface decomposes as  $\Sigma=K\sqcup \Int\,C_1\sqcup...\sqcup\Int\,C_k$, where $K$ is compact, and $C_i$ is a cusp around $p_i$ bounded by a horocycle $H_i$.

\paragraph{Branched projective structures.} A \emph{branched projective structure} on the surface $\Sigma$ is a system of branched projective charts $(D_i,U_i)_{i\in I}$ . It means that $(U_i)_{i\in I}$ is a locally finite cover of $\Sigma$ by open discs, and that the maps $D_i:U_i\to\C\PP^1$ are nonconstant holomorphic maps such that in the intersection of two domains $U_i\cap U_j$, the cocycle relation $D_j=\phi_{ij}\circ D_i$ holds for some Möbius transformation $\phi_{ij}$.

A branched projective structure is, up to projective automorphisms of $\C\PP^1$, determined by a \emph{development-holonomy} pair $(\DD,\rho)$, where $\DD:\Hyp\to\C\PP^1$ is a nonconstant holomorphic map called the \emph{developing map} which globalizes the branched projective charts, and $\rho:\pi_1(\Sigma)\to PSL_2(\C)$ is a morphism called the \emph{holonomy representation} which globalizes the transition functions. Moreover the following equivariance relation holds for every $\gamma\in\pi_1(\Sigma)$:
\begin{equation}
\label{equivariance}
\DD\circ\gamma=\rho(\gamma)\circ\DD.
\end{equation}

\paragraph{Parabolic structures.} Consider a branched projective structure on a hyperbolic surface $\Sigma$ of finite type characterized by a development-holonomy pair $(\DD,\rho)$. Denote by $pr:\Hyp\to\Sigma$ the universal cover of $\Sigma$.

Let $C_i$ be a cusp bounded by a horocycle $H_i$. In the sequel, $\widetilde{C}_i$ denotes a connected component of $pr^{-1}(C_i)$, which is invariant by a parabolic element $\gamma_i\in\pi_1(\Sigma)$ (this is the lift of the translation over the corresponding primitive closed horocycle $H_i$). 

We say that the structure is \emph{parabolic at $C_i$} if there exists $h_i:\widetilde{C}_i\to\C\PP^1$ and $A_i\in PSL_2(\C)$ such that:
\begin{enumerate}
\item $h_i$ is a biholomorphism onto its image;
\item $h_i$ conjugates the actions of $\gamma_i$ and $z\mapsto z+1$, and $A_i$ conjugates the actions of $z\mapsto z+1$ and $\rho(\gamma_i)$;
\item $\DD=A_i\circ h_i$ in restriction to $\widetilde{C}_i$.
\end{enumerate}

We say that the branched projective structure is \emph{parabolic} if it is parabolic at every cusp.

\paragraph{Remark 1.} Let $\widetilde{C}_i,\widetilde{C}_i'$ be two connected components of $pr^{-1}(C_i)$, corresponding to parabolic elements $\gamma_i,\gamma_i'\in \pi_1(\Sigma)$. There exists an element $\gamma\in \pi_1(\Sigma)$ such that $\gamma(\widetilde{C}_i)=\widetilde{C}_i'$ and which conjugates $\gamma_i$ and $\gamma_i'$. We then have $\DD_{|\widetilde{C}_i'}=\rho(\gamma)\circ\DD_{|\widetilde{C}_i}\circ\gamma^{-1}$. This shows that the definition of being parabolic at the cusp $C_i$ does not depend on the choice of a particular lift of $C_i$.

\paragraph{Remark 2.} For a branched projective structure to be parabolic, it is necessary that \emph{the holonomy representation is parabolic}: the image by $\rho$ of a parabolic element $\gamma$ has to be a parabolic matrix of $PSL_2(\C)$. If one prefers, the holonomy over any loop around the punctures has to be conjugated to a translation. As we will see later, this is not sufficient.

\paragraph{Remark 3.} Since $h_i$ conjugates the action of $\gamma_i$ and $z\mapsto z+1$, it sends $\widetilde{C}_i$ inside a half plane bounded by a horizontal line. Up to postcomposition by a Möbius commuting with $z\mapsto z+1$, we can always assume that $h_i(\widetilde{C}_i)\dans \Hyp_{\geq 1}=\lbrace z\in\C;\,\im(z)\geq 1\rbrace$.

\paragraph{Remark 4.} It is convenient to think that a local model for a branched projective structure which is parabolic at a cusp is given by the inclusion $\iota:D\to\C\PP^1$, $D\dans\Hyp_{\geq 1}$ being invariant by $z\mapsto z+1$. Indeed, a structure is parabolic at a cusp $C_i$ if and only if the developing map reads as $\iota$ after holomorphic change of coordinates in $\widetilde{C}_i$ which conjugates the actions of $z\mapsto z+1$ and $\gamma_i$, and a Möbius change of coordinate at the goal which conjugates the actions of $\rho(\gamma_i)$ and $z\mapsto z+1$.

\paragraph{Remark 5.} It is possible to think of the developing map as a multivalued holomorphic map over $\Sigma$. Then, the structure is parabolic at a cusp $C_i$ if in some holomorphic coordinate $z$ in $C_i$, one has $\DD(z)=\frac{1}{2\mathbf{i}\pi}\log(z)$. Non-parabolic structures in a punctured disc with parabolic holonomy representation are given for example by the multivalued holomorphic maps $h_n(z)=\frac{1}{2\mathbf{i}\pi}\log(z)+\frac{1}{z^n}$.\\

Basic examples of parabolic projective structures are given by \emph{uniformization} and more generally by the \emph{covering projective structures}, whose developing maps are covering maps onto their images (see \cite{DD} as well as the references therein). There are also exotic parabolic structures constructed by Hejhal \cite{He} by a surgery process called \emph{grafting} which produce parabolic structures which are not of covering type. Finally, (\cite{CDFG},Lemma 8) provides examples of parabolic projective structures on punctured spheres whose holonomy representations are dense in $PSL_2(\C)$. Using the Schwartzian parametrization of projective structures, it is possible to prove that the space of (non-branched) projective structures on a surface of genus $g$ with $n$ punctures has the structure of a complex affine space of dimension $3g-3+n$ (\cite{DD},Paragraph 6.1.).

\subsection{The foliated geodesic flow of Riccati foliations}
\label{foliatedflows}

\paragraph{Geodesic and horocyclic flows.} The \emph{geodesic flow} of $T^1\Hyp$ is defined by flowing a vector $v$ at unit speed along the geodesic it directs. We denote it by $\widetilde{G}_t$. It is well known that this flow has hyperbolic properties, and that the unstable (resp. stable) manifolds are given by \emph{horocycles} endowed with the outward (resp. inward) unit normal vector field (horocycles are horizontal lines and euclidian circles tangent to the boundary of $\Hyp$). These manifolds foliate $T^1\Hyp$, and the arc length parametrization of the unstable (resp. stable) horocycles gives the \emph{unstable} (resp. \emph{stable}) \emph{horocyclic flow}, denoted by $\widetilde{H}^u_t$ (resp. $\widetilde{H}^s_t$). The unstable and stable foliations are respectively denoted by $\widetilde{\W}^u$ and $\widetilde{\W}^s$, and their leaves, by $\widetilde{W}^u(v)$, $\widetilde{W}^s(v)$, $v\in T^1\Hyp$.

The saturated sets of unstable and stable horospheres by the geodesic flow are respectively called \emph{center-unstable and center-stable manifolds} and are denoted by $\widetilde{W}^{cu}(v)$, $\widetilde{W}^{cs}(v)$, $v\in T^1\Hyp$. They form two foliations of $T^1\Hyp$ by planes called the \emph{center-unstable and center-stable foliations} and denoted by $\widetilde{\W}^{cu}$ and $\widetilde{\W}^{cs}$.

Let $\Sigma$ be a hyperbolic surface of finite type. We can push these flows by the differential of the Riemannian universal cover $pr:\Hyp\to\Sigma$ (which is by definition a local isometry). This defines three flows on $T^1\Sigma$ denoted respectively by $g_t$, $h^u_t$ and $h^s_t$. The invariant foliations will be denoted by $\W^{\ast}$, $\ast=s,u,cs,cu$. All these flows preserve a canonical volume form: the Liouville measure $\Liouv$ (normalized in such a way that $\Liouv(T^1\Sigma)=1$). Moreover, the famous theorem of Hopf \cite{Ho} asserts that the geodesic flow $g_t:T^1\Sigma\to T^1\Sigma$ is ergodic with respect to the Liouville measure.

\paragraph{Riccati foliations.} Given a representation $\rho:\pi_1(\Sigma)\to PSL_2(\C)$, there is an associated \emph{Riccati foliation} obtained by suspension of the action on $\C\PP^1$.

More precisely, $\pi_1(\Sigma)$ acts diagonally on $\Hyp\times\C\PP^1$: the action on the first factor is by deck transformations, and that on the second one is by $\rho$. By taking the quotient, we obtain a manifold $M$, called the \emph{suspended manifold} endowed with:
\begin{itemize}
\item a fiber bundle $\Pi:M\to\Sigma$, whose fibers $F_p=\Pi^{-1}(p)$, $p\in\Sigma$ are copies of $\C\PP^1$;
\item a \emph{suspended foliation} $\F$ transverse to the fibers of $\Pi$, whose leaves are covers of $\Sigma$, and with holonomy representation $\rho$.
\end{itemize}

The data of $(\Pi,M,\Sigma,\C\PP^1,\F)$ will be called a \emph{Riccati foliation}. As was mentioned in the Introduction, these foliations can be defined by \emph{Riccati equations} on closed algebraic curves.

It is then possible to consider a Riemannian metric on $M$, that we will call \emph{admissible}, and which satisfies:
\begin{itemize}
\item for the induced metric on each leaf $L$, the restriction $\Pi_{|L}:L\to\Sigma$ is a Riemannian cover (in particular, all leaves are hyperbolic);
\item the induced metric on each fiber is compatible with its conformal structure, i.e. after a conformal change of coordinates it is the usual \emph{Fubini-Study metric} given by:
$$ds^2=\frac{dx^2+dy^2}{(1+x^2+y^2)^2}.$$
\item leaves and fibers are orthogonal.
\end{itemize}

If we have moreover a developing map $\DD$, we can define a holomorphic section $\sigma^0:\Sigma\to M$ of the bundle, called the \emph{diagonal section}, which is transverse to $\F$ except at a finite number of points, and which is induced by equivariance by  $(Id,\DD):\Hyp\to\Hyp\times\C\PP^1$.

\paragraph{Foliated flows.}  In the sequel we consider a Riccati foliation $(\Pi,M,\Sigma,\C\PP^1,\F)$ endowed with an admissible metric. Since in restriction to the leaves, the fibration is a local isometry, its differential induces a fiber bundle $\Pi_{\ast}:T^1\F\to T^1\Sigma$, where $T^1\F$ is the \emph{unit tangent bundle of the foliation}, i.e. the set of unit tangent vectors tangent to the leaves of $\F$. This fiber bundle is foliated by the $T^1L$, $L$ leaf of $\F$. We denote by $\hcF$ this foliation. The fiber $\Pi_{\ast}^{-1}(v)$ shall be denoted by $F_{\ast,v}$.

Since any leaf $L$ is uniformized by the Poincaré plane $\Hyp$, $T^1L$ carries a geodesic flow and two horocyclic flows. Hence we have three flows of $T^1\F$ which, when restricted to a leaf $T^1L$, coincide with its geodesic and horocyclic flows. We call them the \emph{foliated geodesic flow}, denoted by $G_t$, the \emph{foliated unstable horocyclic flow}, denoted by $H_t^u$, and the \emph{foliated stable horocyclic flow}, denoted by $H_t^s$

\paragraph{Projective cocycles.} Bonatti, G\'omez-Mont and Vila \cite{BGVil} remarked that these foliated flows produce \emph{locally constant projective cocycles}. Indeed, since all leaves are Riemannian covers of the base, the foliated geodesic flow projects down to the geodesic flow of $T^1\Sigma$. Hence, it sends fibers to fibers. We shall denote the resulting cocycle by:
$$A_t(v)=(G_t)_{|F_{\ast v}}:F_{\ast, v}\To F_{\ast, g_t(v)},$$
for $t\in\R$, and $v\in T^1\Sigma$. The term cocycle refers to the following formula:
$$A_{t_1+t_2}(v)=A_{t_1}(g_{t_2}(v))A_{t_2}(v).$$

If we choose any orbit segment $c=g_{[0;t]}(v)$, then $A_t(v)$ is the holonomy map along the path $c$. Hence, the foliated geodesic flow sends fibers to fibers as a projective transformation. The cocycle is locally constant because the bundle is flat: in particular, if two orbit paths $c=g_{[0;t]}(v)$ and $c'=g_{[0;t']}(v')$ are covered by a same chain of trivializing charts, then $A_t(v)$ and $A_{t'}(v')$ are equal as projective transformations of $\C\PP^1$.

In the same way, the foliated unstable and stable horocyclic flows produce cocycles:
$$B_t^u(v)=(H^u_t)_{|F_{\ast v}}:F_{\ast, v}\To F_{\ast, h^u_t(v)},$$
$$B_t^s(v)=(H^s_t)_{|F_{\ast v}}:F_{\ast, v}\To F_{\ast, h^s_t(v)},$$

\paragraph{Lyapunov exponents.} When the representation is parabolic, Bonatti, G\'omez-Mont and Vila \cite{BGVil} proved that Oseledets' theorem can be applied to the cocycle $A_t$ because the following integrability condition holds:
$$\int_{T^1\Sigma}\log^+||A_{\pm 1}(v)||\,d\Liouv(v)<\infty.$$
A consequence is the existence of \emph{Lyapunov exponents} for $\Liouv$-almost every $v\in T^1\Sigma$:
$$\lambda^+(v)=\lim_{t\to\infty}\frac{1}{t}\log||A_t(v)||,$$
$$\lambda^-(v)=\lim_{t\to\infty}\frac{1}{t}\log||A_t(v)^{-1}||^{-1}.$$

By ergodicity of Liouville measure these quantities are constant on a full and invariant set: we call $\lambda^+,\lambda^-$ these numbers. Remark that $\lambda^+=-\lambda^-$.

If $\Sigma$ were compact, the following theorem would be attributed to Bonatti, G\'omez-Mont and Viana \cite{BGV}. Since it is not compact, it is a consequence of the work of Avila-Viana, and of the coding of the geodesic flow.

More precisely, using the Bowen-Series coding of the action of the surface group \cite{BS}, Series was able to prove that the geodesic flow of $\Sigma$ is a \emph{sophic system} \cite{Se}. For our purpose, the modification of this coding by Ledrappier and Sarig \cite{LS} will be more adapted.
They provide a geometric \emph{Markov partition} with countably many symbols for the geodesic flow on $T^1\Sigma$ and locally Hölder height function. They also provide the symbolic description of the Liouville measure and prove that it has a consistent local product structure with uniformly log-bounded densities in the local stable and unstable sets (this is Lemma 3.1 of \cite{LS}).

In \cite{AV}, the authors give a sufficient condition for a locally constant projective cocycle over a Markov map with countably many symbols endowed with an ergodic probability measure with the local product structure to have a simple Lyapunov spectrum. A combination of these works gives the following:

\begin{theorem}
\label{positivitylyapexp}
Let $(\Pi,M,\Sigma,\C\PP^1,\F)$ be a Riccati foliation endowed with an admissible metric. Assume that the holonomy representation $\rho:\pi_1(\Sigma)\to PSL_2(\C)$ is parabolic. Then the following dichotomy holds true:
\begin{itemize}
\item either there exists a probability measure on $\C\PP^1$ invariant by the holonomy group $\rho(\pi_1(\Sigma))$;
\item or one has $\lambda^+(v)>0$ for Liouville-almost every $v\in T^1\Sigma$.
\end{itemize}
\end{theorem}

\paragraph{The Lyapunov sections.} Assume that there is no holonomy invariant measure and that the holonomy representation is parabolic. Then, by Theorem \ref{positivitylyapexp}, and Oseledets' theorem, we have that:

\begin{proposition}
\label{lines}
For Liouville almost every $v\in T^1\Sigma$ there exists a splitting of the linear fiber $\widetilde{F}_{\ast,v}=\sigma^+(v)\oplus\sigma^-(v)$ such that:
\begin{enumerate}
\item it varies measurably with the point $v$;
\item it commutes with the cocycle: for every $t\in\R$, $A_t(v)\sigma^{\pm}(v)=\sigma^{\pm}(g_t(v))$;
\item we have the following property of attraction:
$$\lim_{t\to\infty}\frac{1}{t}\log\,\dist_{F_{\ast,v}}(A_t(v)w,\sigma^+(g_t(v)))=-2\lambda^+\,\,\,\,\,for\,\,all\,\,w\in F_{\ast,v}\moins\{\sigma^-(v)\},$$
$$\lim_{t\to\infty}\frac{1}{t}\log\,\dist_{F_{\ast,v}}(A_{-t}(v)w,\sigma^-(g_{-t}(v)))=-2\lambda^+\,\,\,\,\,for\,\,all\,\,w\in F_{\ast,v}\moins\{\sigma^+(v)\}.$$
\item the sections are determined by the following properties:
$$\lim_{t\to\infty} ||A_{-t}(v)x||=0\,\,\,\,if\,\,and\,\,only\,\,if\,\,x\in \sigma^+(v),$$
$$\lim_{t\to\infty} ||A_{t}(v)x||=0\,\,\,\,if\,\,and\,\,only\,\,if\,\,x\in \sigma^-(v).$$
\end{enumerate}
\end{proposition}

\paragraph{Remark.} 

In the third assertion, $\sigma^{\pm}(v)$ are thought as elements of the projective fiber $F_{\ast,v}$. In the last assertion, we see $A_{\pm t}(v)$ as an element of $SL_2(\C)$ acting on a copy of $\C^2$: recall that the bundle is supposed to be linearizable.\\

The two subspaces defined in the proposition above can be thought as elements of the projective fiber $F_{\ast,v}$. Therefore we have two measurable sections $\sigma^{\pm}$ of the bundle $\Pi_{\star}$ which are called \emph{Lyapunov sections}.

The following proposition is due to Bonatti and G\'omez-Mont \cite{BG}: it relies on the fourth assertion stated in Proposition \ref{lines}.

\begin{proposition}
\label{sections}
\begin{enumerate}
\item The two Lyapunov sections commute with the geodesic flows: $G_t\circ\sigma^{\pm}=\sigma^{\pm}\circ g_t$.
\item The section $\sigma^+$ commutes with the unstable horocyclic flows: $H_t^u\circ\sigma^{+}=\sigma^{+}\circ h_t^u$.
\item The section $\sigma^-$ commutes with the stable horocyclic flows: $H_t^s\circ\sigma^{-}=\sigma^{-}\circ h_t^s$.
\end{enumerate}
\end{proposition}

\section{Distribution of the singularities}
\label{distributionsingularities}

The purpose of this section is to prove Theorem \ref{distriblimit} and to give several description of the distribution of limit points of developed geodesic rays in the Riemann sphere. This also gives the statistical distribution of the singularities of holonomy germs from a fiber to the image of the holomorphic section $\sigma^0$ along almost every developed ray. 

The results of this section are consequences of Theorem \ref{limitedprojective}, which we assume to hold for the moment. In the sequel $\Sigma$ will stand for a hyperbolic surface of finite type, and $(\DD,\rho)$ for a parabolic branched projective structure. We assume that the holonomy group $\rho(\pi_1(\Sigma))$ has no invariant probability measure on $\C\PP^1$. Under these hypothesis, Theorem \ref{limitedprojective} implies that for every $p\in\Sigma$ and $d\theta$-almost every $v\in T^1_p\Sigma$ we have (with an obvious abusive notation):
$$\lim_{t\to\infty}\DD(g_t(v))=\sigma^-(v).$$

\subsection{Disintegration of the SRB measure of the foliated geodesic flow}

\paragraph{Distribution of the singularities.}
Unit vectors tangent to $\Sigma$ are distributed uniformly according to the Liouville measure. Denote by $(d\theta_p)_{p\in\Sigma}$ the family conditional measures of the Liouville measure on the unit tangent fibers $T^1_p\Sigma$ with respect to the area element of $\Sigma$. Hence the limits of developed geodesic rays in the fiber of a point $p$ are distributed according to:
\begin{equation}
\label{distributionlimitpoints}
s_p=\sigma^-\ast(d\theta_p).
\end{equation}

\paragraph{Remark.} The probability measures $s_p$, $p\in\Sigma$ are quasi-invariant by holonomy maps of the foliation $\F$: this is another way to say that there is a well defined \emph{measure class} on $\C\PP^1$ which describes the distribution of limits of developed geodesic rays. This is so because the measure class of $d\theta_p$, as well as the section $\sigma^-$ are invariant by center-stable holonomies (which are smooth since the stable horocyclic flow is smooth).

\paragraph{The SRB measure.} \cite{BGVil} proved that the foliated geodesic flow $G_t$ possesses a unique \emph{SRB measure}: it possesses a probability measure $\mu^+$ whose basin of attraction has full volume. It follows from the third property stated in Proposition \ref{lines} that this measure is precisely described by:
$$\mu^+=\sigma^+\ast\Liouv.$$

The flow $G_{-t}$ also possesses a unique SRB measure which is precisely described by:
$$\mu^-=\sigma^-\ast\Liouv.$$

Notice that even if these measure are singular (since the sections $\sigma^+$ and $\sigma^-$ are disjoint almost everywhere) their projections via the canonical map $pr:T^1\F\to M$ are equal: we denote it by $m$. Indeed $\mu^-$ is exactly the image of $\mu^+$ by the involution $v\in T^1\F\mapsto -v$.

\paragraph{Disintegration in the fibers.} Here we show that the conditional measures $m_p$ of $m$ on the fibers coincide with the distribution $s_p$ defined by Formula \ref{distributionlimitpoints}. The proof follows the lines of that of Theorem F of \cite{Al3} where a similar result is stated in the context of Gibbs measures of the foliated geodesic flow (although this theorem is stated for a compact base its proof can be copied without modification). The main idea is that the conditional measures of $\mu^-$ in the fibers $F_{\ast,v}$ are given by the Dirac masses at $\sigma^-(v)$ and that we obtain conditional measures of $m$ in a fiber $F_p$ by integration of those of $\mu^-$ on fibers of unit vectors tangent to $p$.

\begin{proposition}
\label{SRB}
Let $m$ be the projection of the unique SRB measure of the geodesic flow and $(m_p)_{p\in\Sigma}$ be its system of conditional measures on the fibers. Then for every $p\in\Sigma$,
$$m_p=s_p.$$
\end{proposition}

\subsection{Foliated harmonic measure and its discretization}

\paragraph{The unique foliated harmonic measure.} Each leaf $L$ is endowed with a Laplace operator $\Delta_L$ which generates a one-parameter semi-group called the \emph{heat diffusion}, characterized by a \emph{heat kernel} $p(t;x,y)$. This allows for every $x\in L$ to define the \emph{Wiener probability measure} on the space $\Omega_x$ of continuous paths $\omega:[0;\infty)\to L$ starting at $x$, that will be denoted by $W_x$. It has the Markov property, and projects down to the heat density $p(t;x,y)dy$ by the map $\omega\mapsto\omega(t)$. A \emph{Brownian path} starting at $x$ is a typical path for $W_x$.

\emph{Foliated harmonic measures} for $\F$ are measures on $M$ which are invariant by the \emph{leafwise heat diffusion operator} (which by definition induces on every leaf $L$ its heat diffusion operator). They have been considered by Garnett \cite{Gar} in the context of \emph{compact} foliated manifolds. In our context the existence of such measure is guaranteed by the Main Theorem of \cite{Al1}.

\begin{proposition}
\label{uniquenessharmonic}
Under the hypothesis of Theorem \ref{limitedprojective}, there exists only one foliated harmonic measure.
\end{proposition}

\begin{proof}
This can be deduced from the Main Theorem of \cite{Al1} (which gives a bijective correspondence between harmonic and stationary measure for a probability measure on the holonomy group that we shall describe below), from Furstenberg's theorem (ensuring the uniqueness of stationary measures in the present context under some integrability conditions \cite{Fu1}) and from Section 3.4 of the first author's PhD thesis \cite{Al4} (which shows the integrability conditions under the hypothesis of Theorem \ref{limitedprojective}).
\end{proof}

\begin{proposition}
\label{projectionSRB}
Under the hypothesis of Theorem \ref{limitedprojective}, the projection $m$ of the unique SRB measure of the foliated geodesic flow is the unique foliated harmonic measure for $\F$.

In particular the family $(s_p)_{p\in\Sigma}$ defined by Formula \ref{distributionlimitpoints} is the family of conditional measures of the unique foliated harmonic measure for $\F$.
\end{proposition}

\begin{proof}
The measure $\mu^+$ is invariant by the joint action of the foliated geodesic and unstable horocyclic flow (see Proposition \ref{lines}). Hence its projection is a harmonic measure for $\F$: see the proofs of \cite{Al2,Ma} made in the compact case but which are still valid in our context.
\end{proof}

\paragraph{Discretization.} Given a probability measure $\nu$ on the fundamental group $\pi_1(\Sigma)$ we say that a measure $s$ on the Riemann sphere $\C\PP^1$ is $\nu$-\emph{stationary} if:
$$s=\sum_{\gamma\in\pi_1(\Sigma)}\nu(\gamma)\rho(\gamma)\ast s.$$

The \emph{discretization of the Brownian motion} performed by Furstenberg-Lyons-Sullivan \cite{Fu2,LySu} provides a bijective correspondence between foliated harmonic measures and stationary measures for the action of the holonomy group on the fiber (see \cite{Al1}).

In our context it yields a family $(\nu_z)_{z\in\Hyp}$ of probability measures on $\pi_1(\Sigma)$ with full support and equivariance property $\gamma\ast\nu_z=\nu_{\gamma z}$ (hence it defines a family $(\nu_p)_{p\in\Sigma}$ on $\pi_1(\Sigma)$) such that the conditional measure of the unique harmonic measure on the fiber $F_p\simeq\C\PP^1$ is precisely the unique $\nu_p$-stationary measure (see the proof of Proposition \ref{uniquenessharmonic}). This provides another characterization of the distribution of limit points of images of most geodesic rays by the developing map:

\begin{proposition}
\label{stationarymeasures}
Assume that the hypothesis of Theorem \ref{limitedprojective} hold. Let $(\nu_p)_{p\in\Sigma}$ be a family of measures given by the Furstenberg-Lyons-Sullivan procedure of discretization of the Brownian motion. Then for every $p\in\Sigma$ $s_p$ coincides with the unique $\nu_p$-stationary measure on the fiber $F_p$.
\end{proposition}

\subsection{Family of harmonic measures on the Riemann sphere}

\paragraph{Limits of developed Brownian paths.} In \cite{H} the second author adopts the point of view of Brownian motion and studies images by the developing map of Brownian paths (that we also call \emph{developed Brownian paths}). The main Theorem of \cite{H} is:

\begin{theorem}
\label{hussenot1}
Let $\Sigma$ be a hyperbolic surface of finite type. Consider a branched projective structure on $\Sigma$ represented by a development-holonomy pair $(\DD,\rho)$, that is parabolic and non-elementary. Then the following dichotomy holds:
\begin{itemize}
\item if $\DD$ is not onto, then for all $z\in\Hyp$ and almost every Brownian path $\omega$ starting at $z$, there exists $\e(\omega)\in\C\PP^1$ such that $\DD(\omega(t))$ converges to $\e(\omega)$ when $t$ goes to $\infty$;
\item if $\DD$ is onto,  then for all $z\in\Hyp$ and almost every Brownian path $\omega$ starting at $z$, the path $\DD(\omega(t))$ does not have any limit when $t$ goes to $\infty$.
\end{itemize}
\end{theorem}

\paragraph{Asymptotic behaviour of developed Brownian paths.} Although in the second case of Theorem \ref{hussenot1} a developed Brownian path does not have a limit we can describe its asymptotic behaviour. \emph{Almost every developed Brownian path spends most of its time very close to some point of $\C\PP^1$, and the distribution of these points is exactly given by the distribution of limits of developed geodesic rays}.

\begin{theorem}
\label{hussenot2}
Let $\Sigma$ be a hyperbolic surface of finite type. Consider a branched projective structure on $\Sigma$ represented by a development-holonomy pair $(\DD,\rho)$, that is parabolic and non-elementary. Then for every $z\in\Hyp$ and almost every Brownian path $\omega$ starting from $z$, there exists a point $\e(\omega)\in \C\PP^1$ such that:
$$\lim\limits_{t\to\infty}\frac{1}{t}\int_0^t\DD*\delta_{\omega(s)}ds=\delta_{\e(\omega)}$$
\end{theorem}

Denote by $\e_z$ the distribution of $\e(\omega)$ subject to the condition $\omega(0)=z$: by equivariance of $\DD$ this distribution satisfies the equivariance relation $\gamma\ast\e_z=\e_{\gamma z}$ for every $\gamma\in\pi_1(\Sigma)$. In particular it induces a family of probability measures on the fibers $F_p$ $(\e_p)_{p\in\Sigma}$. The proof of Theorem \ref{hussenot2} in \cite{H} provides more information about $\e_p$: it coincides exactly with the unique $\nu_p$-stationary measure (recall that $\nu_p$ is given by the discretization of the Brownian motion). As a consequence, we find that:

\begin{proposition}
\label{asydistribbrownian}
Assume that the hypothesis of Theorem \ref{limitedprojective} hold. Let $(\e_p)_{p\in\Sigma}$ be the family describing the asymptotic behaviour of developed Brownian path defined in Theorem \ref{hussenot2}. Then for every $p\in\Sigma$, the measure $s_p$ coincide with $\e_p$.
\end{proposition}

\paragraph{Family of harmonic measures.} It is classical that by considering the exit distribution of the Brownian motion on an open set, one obtains a family of measures, \emph{the harmonic measures}, on its boundary which is used to solve the Dirichlet problem of finding harmonic functions with prescribed boundary conditions.

It is also possible to associate to any non-elementary representation $\rho:\pi_1(\Sigma)\to PSL_2(\C)$ a family of harmonic measures. Namely it has been shown by Furstenberg \cite{Fu3} that for such a representation, there exists a unique (up to a null Borel set for the Lebesgue measure) measurable map $\beta:\R\PP^1\to\C\PP^1$ which is $\rho$-equivariant (here the action of $\pi_1(\Sigma)$ on $\R\PP^1$ is the natural extension of that on $\Hyp$ which is given by the uniformization). This map is called the \emph{Furstenberg's boundary map}.

By pushing by the Furstenberg's boundary map $\beta$ the measure on $\R$ whose density with respect to Lebesgue is given by the Poisson kernel, Deroin and Dujardin proved the following:

\begin{proposition}
\label{harmonicmeasures}
Let $\Sigma$ be a hyperbolic surface of finite type and $\rho:\pi_1(\Sigma)\to PSL_2(\C)$ be a representation which preserves no measure on $\C\PP^1$. Then there exists a unique family of probability measures $(\theta_z)_{z\in\Hyp}$ which verifies the following properties:
\begin{enumerate}
\item it is equivariant: for every $z\in\Hyp$ and $\gamma\in\pi_1(\Sigma)$ we have $\rho(\gamma)\ast\theta_z=\theta_{\gamma z}$;
\item it is harmonic: for every Borel set $A\dans\C\PP^1$ the map $z\mapsto \theta_z(A)$ is harmonic for the Laplace operator.
\end{enumerate}
\end{proposition}

Once again the equivariance allows us to define a family of measures on the fibers $(\theta_p)_{p\in\Sigma}$. Deroin and Dujardin note that the family $(\e_z)_{z\in\Hyp}$ given by Theorem \ref{hussenot2} satisfies these conditions. Hence we can conclude the proof of Theorem \ref{distriblimit}.

\begin{proposition}
Assume that the hypothesis of Theorem \ref{limitedprojective} hold. Let $(\theta_p)_{p\in\Sigma}$ be the family of harmonic measures of the representation $\rho$ defined in Proposition \ref{harmonicmeasures}. Then for every $p\in\Sigma$, the measure $s_p$ coincide with $\theta_p$.
\end{proposition}

\subsection{The set of singularities is uncountable and dense in the limit set}

\paragraph{The distribution is non atomic.} We have the following lemma which will imply that the set of limits of developed geodesic rays is uncountable.

\begin{lemma}
\label{noatom}
Assume that the hypothesis of Theorem \ref{limitedprojective} hold. Then the distribution $s_p$ is non atomic.
\end{lemma}

\begin{proof}
The easiest way to see this fact is to use Proposition \ref{stationarymeasures}: $s_p$ is the unique $\nu_p$-stationary measure. But a classical argument shows that since the holonomy group does not preserve any measure this measure has to be non atomic for if the contrary were true we could consider the finite subset $X\dans\C\PP^1$ of atoms of greatest mass. For $x\in X$ we would obtain by stationarity:
$$s_p(x)=\sum_{\gamma\in\pi_1(\Sigma)}\nu_p(\gamma) s_p(\rho(\gamma)^{-1}x),$$
which would imply $s_p(\rho(\gamma)^{-1}x)=s_p(x)$ for every $\gamma\in\pi_1(\Sigma)$. Finally the set $X$ would be invariant by holonomy which contradicts the hypothesis.
\end{proof}

\begin{proposition}
\label{uncountable}
Assume that the hypothesis of Theorem \ref{limitedprojective} hold. Then for every $p\in\Sigma$, the set of $\sigma^-(v)$ where $v$ ranges a full $d\theta$-measure subset of $T^1_p\Sigma$ is uncountable.
\end{proposition}

\begin{proof}
This follows directly from Lemma \ref{noatom} and the fact that any measure supported on a countable set has atoms.
\end{proof}

\paragraph{The distribution charges open sets of the limit set.} By hypothesis, the holonomy group $\rho(\pi_1(\Sigma))$ is a non-elementary subgroup of $PSL_2(\C)$: it possesses a unique minimal set $\Lambda_{\rho}$ called its \emph{limit set}. We can show the following proposition:

\begin{proposition}
Assume that the hypothesis of Theorem \ref{limitedprojective} hold. Then for every $p\in\Sigma$, the set of $\sigma^-(v)$ where $v$ ranges a full $d\theta$-measure subset of $T^1\Sigma$ is dense in the limit set.
\end{proposition}

\begin{proof}
Firstly all $\sigma^-(v)$ belong to the limit set of the holonomy group because of the invariance by the foliated geodesic flow: we can always write $\sigma^-(v)=(A_t)^{-1}(v)\sigma^-(g_t(v))$ and for almost every $v$, $A_t(v)$ is a word in the generators of $\rho(\pi_1(\Sigma))$ whose length goes to infinity with $t$.

Secondly for every $p\in\Sigma$, $v\in T^1_p\Sigma$ and $\gamma\in\pi_1(\Sigma)$, $\rho(\gamma)\sigma^-(v)$ belongs to the image of the restriction of $\sigma^-$ to $T^1_p\Sigma$: by minimality of the action of the holonomy group on its limit set this implies that the image of the restriction of $\sigma^-$ to $T^1_p\Sigma$ is dense in the limit set. In order to see this fact, work in the universal cover of $\Sigma$. There is a identification $T^1\Hyp\simeq\Hyp\times\R\PP^1$ obtained by associating to a vector $v$ the couple $(z,\xi)=(c_v(0),c_v(\infty))$ where $c_v$ is the geodesic directed by $v$. This identification trivializes the center-stable foliation (we will also meet in the sequel a different identification which trivializes the center-unstable foliation) and conjugates the natural actions of $\pi_1(\Sigma)$ on these two spaces. Since the section $\sigma^-$ commutes with the center-stable foliations, its lift can be written in coordinates as:
$$\widetilde{\sigma}^-(z,\xi)=(z,\xi,\widetilde{s}^-(\xi)).$$
Now the equivariance relation $\rho(\gamma)\widetilde{s}^-(\xi)=\widetilde{s}^-(\gamma\xi)$ shows that for all $v\in T^1_p\Sigma$, $\rho(\gamma)\sigma^-(v)\in\sigma^-(T^1_p\Sigma)$. That concludes the proof.
\end{proof}

\section{Limits of developed geodesic rays}
\label{limitsgeodesicrays}

\subsection{Distance between diagonal and Lyapunov sections along geodesics}

Until the end of this article, $\Sigma$ is a hyperbolic surface of finite type, and $(\DD,\rho)$ is a parabolic branched projective structure. We consider the associated Riccati foliation $(\Pi,M,\Sigma,\C\PP^1,\F)$, the associated diagonal section $\sigma^0$, and we endow $M$ with an admissible metric. Furthermore, we assume that the holonomy group $\rho(\pi_1(\Sigma))$ has no invariant probability measure on $\C\PP^1$.

The diagonal section $\sigma^0$ clearly induces a smooth section of the unit tangent bundle, that we also denote by $\sigma^0:T^1\Sigma\to T^1\F$ which is invariant by the holonomy over the unit tangent fibers $T^1_p\Sigma$. We also call it the diagonal section of $\Pi_{\ast}$.

\paragraph{Developing map and the cocycle.} Recall that by definition, a \emph{developed geodesic ray} in $\C\PP^1$ is the image of a geodesic ray of $\Hyp$ by the developing map $\DD$. We want to prove that in that case, a typical developed geodesic ray has a limit.

Recall moreover that for every $v\in T^1\Sigma$, $A_t(v)$ is the holonomy map along the orbit segment $g_{[0;t]}(v)$. Hence, we have the following important formula which holds (with the obvious abusive notation $\DD(g_t(v))=\DD(\widetilde{c}_v(t))$ for the lift $\widetilde{c}_v(t)$ of the geodesic directed by $v$) for every $t\in\R$:
\begin{equation}
\label{formuladeveloping}
\DD(g_t(v))=(A_t(v))^{-1}\sigma^0(g_t(v)).
\end{equation}

\paragraph{North-South dynamics.} By definition of Lyapunov sections (see Proposition \ref{lines}), the fiberwise dynamics of the cocycle over a geodesic orbit is nothing but a North-South dynamics. More precisely, a simple application of the $\eps$-reduction theorem of Oseledets-Pesin (see \cite{KH}) implies the following useful proposition.

\begin{proposition}
\label{northsouth}
Let $(\Pi,M,\Sigma,\C\PP^1,\F)$ be a Riccati foliation with a parabolic holonomy representation which preserves no measure on $\C\PP^1$. Let $0<\lambda_1<\lambda_2<\lambda_+$. Then for Liouville-almost every $v\in T^1\Sigma$, there exists $T_0$ such that for every $t\geq T_0$, we have:
$$A_t(v)^{-1}\left[\,^cD(\sigma^+(g_t(v)),e^{-\lambda_1t})\right]\dans D(\sigma^-(v),e^{-\lambda_2 t}),$$
where $D(x,r)$ denotes the disc in $\C\PP^1$ centered at $x$ and of radius $r$ for the Fubini-Study metric.
\end{proposition}

Henceforth, by Formula \eqref{formuladeveloping} and Proposition \ref{northsouth}, if we want to prove that for Liouville almost every $v\in T^1\Sigma$,
\begin{equation}
\label{limitedevrayae}
\lim_{t\to\infty}\DD(g_t(v))=\sigma^-(v),
\end{equation}
it is enough to prove the following key proposition:

\begin{proposition}
\label{diagonaloutsidelyap}
Let $(\DD,\rho)$ be a non-elementary parabolic branched projective structure on a hyperbolic surface of finite type $\Sigma$. Let $(\Pi,M,\Sigma,\C\PP^1,\F)$ be the associated Riccati foliation and $\sigma^0$ be the associated diagonal section. Then there exists a Borel set $\X\dans T^1\Sigma$ $g_t$-invariant and full for the Liouville measure such that for every $0<\lambda_1
<\lambda_+$, and every $v\in\X$, there exists $T_1$ such that for every $t\geq T_1$, we have:
$$\sigma^0(g_t(v))\notin D(\sigma^+(g_t(v)),e^{-\lambda_1 t}).$$
\end{proposition}

\paragraph{Proposition \ref{diagonaloutsidelyap} implies Theorem \ref{limitedprojective}.} Until the end of this paragraph, we assume that Proposition \ref{diagonaloutsidelyap} holds true. Let us state what remains to be proven. As we mentioned before, this proposition implies that \eqref{limitedevrayae} holds almost everywhere for the Liouville measure. In other words, it implies that the conclusion of Theorem \ref{limitedprojective} holds only for $\Leb$-almost every $z\in\Sigma$. It remains to prove that it holds \emph{for every} $z$. Before we begin the proof of the theorem, let us make some remarks.

\begin{enumerate}
\item Even though the section $\sigma^-$ is a priori only defined on the $g_t$-invariant set full for the Liouville measure $\X$, we know by Proposition \ref{sections} that it commutes with geodesic and stable horocyclic flows: it is well defined on the whole center-stable manifold of every point of $\X$.
\item Since $\sigma^-$ commutes with the geodesic flows, if the conclusion of Theorem \ref{limitedprojective} holds for a vector $v\in T^1\Sigma$, it also holds for every $g_t(v)$, $t\in\R$.
\item All center-stable manifolds, except those of periodic orbits and those corresponding to the cusps, are planes. In particular for every $z\in\Sigma$ and $d\theta$-almost every $v\in T^1_z\Sigma$, the center-stable manifold of $v$ is simply connected.
\item For every $z\in\Sigma$ and $d\theta$-almost every $v\in T^1_z\Sigma$, there exists  $v'\in\X$ such that $v'\in W^{cs}(v)$ and $\dist_{cs}(v,v')<1/2$.
\end{enumerate}

Hence, it is enough to prove that for every $v'\in\X$ whose center-stable manifold is simply connected and every $v\in W^s(v')$ with $\dist_s(v,v')<1$, we have:
$$\lim_{t\to\infty}\DD(g_t(v))=\sigma^-(v).$$

In order to do so, we will need the following proposition, whose proof is postponed until the next paragraph.

\begin{proposition}
\label{sectionlipschitz}
Let $(\DD,\rho)$ be a parabolic branched projective structure on a hyperbolic surface $\Sigma$ of finite type. Then the section $\sigma^0:\Sigma\to M$ is Lipschitz.
\end{proposition}

Now consider $v'\in\X$ whose center-stable manifold is simply connected, as well as $v\in W^s(v')$ such that $v'=h^s_{\delta}(v)$ with $0<\delta<1$. Then, for every $t>0$, $\dist_s(g_t(v),g_t(v'))=\delta e^{-t}$. Moreover, since the center-stable manifold is simply connected, the following conjugacy formula holds for every $t>0$:

\begin{equation}
\label{conjugacyformula}
A_t(v)=(B_{\delta e^{-t}}^s(g_t(v)))^{-1} A_t(v') B_{\delta}^s(v).
\end{equation}
In particular, this shows that $A_t(v)^{-1}\sigma^0(g_t(v))=(B_{\delta}^s(v))^{-1} A_t(v')^{-1} B_{\delta e^{-t}}^s(g_t(v))\,\sigma^0(g_t(v))$.

Since $\sigma^0$ is Lipschitz, there exists $C>0$ such that for every $t>0$, we have:
$$\dist_M(\sigma^0(g_t(v)),\sigma^0(g_t(v')))\leq Ce^{-t}.$$
Moreover, for every $t>0$, 
$$\dist_M(\sigma^0(g_t(v)),B^s_{\delta e^{-t}}(g_t(v))\sigma^0(g_t(v)))\leq\delta e^{-t}.$$
Hence by the triangular inequality, we have for every $t>0$:
$$\dist_M(\sigma^0(g_t(v')),B_{\delta e^{-t}}^s(g_t(v))\,\sigma^0(g_t(v)))\leq (C+\delta)e^{-t}.$$

Recall that $v'\in\X$: we can apply Proposition \ref{diagonaloutsidelyap}, and if we have chosen $0<\lambda_1<\lambda_2<\Min(\lambda^+,1)$, we get $T_2>0$ such that for every $t\geq T_2$:
$$B_{\delta e^{-t}}^s(g_t(v))\,\sigma^0(g_t(v))\notin D(\sigma^+(g_t(v')),e^{-\lambda_1t}).$$

From Proposition \ref{northsouth}, we deduce that for every $t\geq T_2$:
$$A_t(v')^{-1} B_{\delta e^{-t}}^s(g_t(v))\,\sigma^0(g_t(v))\in D(\sigma^-(v'),e^{-\lambda_2t}).$$

Finally, we use the fact that the map $(B^s_{\delta}(v))^{-1}:F_{\ast,v'}\to F_{\ast,v}$ is Lipschitz to prove the existence of $C'>0$ such that for all $t\geq T_2$:
$$\dist_{\C\PP^1}(A_t(v)^{-1}\sigma^0(g_t(v)),\sigma^-(v))\leq C' \dist_{\C\PP^1}(A_t(v')^{-1} B_{\delta e^{-t}}^s(g_t(v)),\sigma^-(v'))\leq C' e^{-\lambda_2 t}.$$

It proves in particular that $\lim_{t\to\infty}\DD(g_t(v))=\sigma^-(v)$. Hence, assuming Proposition \ref{diagonaloutsidelyap} Theorem \ref{limitedprojective} is proven. \quad \hfill $\square$

\paragraph{Proof of Proposition \ref{sectionlipschitz}.} Before giving the proof of the proposition, let us recall a consequence of the uniformization theorem. For a puncture $p_i$ of $\Sigma$, there is a distinguished local holomorphic coordinate $z$ around $p_i$ with $z(p_i)=0$ where the metric reads as:
\begin{equation}
\label{standard}
ds^2=\frac{|dz|^2}{(|z|\,\log|z|)^2}.
\end{equation}

The next lemma asserts that holomorphic changes of coordinate near punctures of $\Sigma$ are close to be hyperbolic isometries.
\begin{lemma}
\label{hologerm}
Let $\D^*$ be the unit disc punctured at the origin, endowed with the complete hyperbolic metric given by \eqref{standard}. Let $h:(\D^*,0)\to(\D^*,0)$ be a germ of biholomorphism fixing the origin. Then $h^*(ds^2)$ and $ds^2$ are conformally equivalent with a conformal factor which tends to $1$ at the origin.
\end{lemma}

\begin{proof}
Write the Taylor expansion at the origin of the germ $h$ as $\sum_{n=1}^{\infty} a_nz^n$ with $a_1\neq 0$. The desired conformal factor is precisely given by:
$$\fhi(z)=|h'(z)|\frac{|z|\,\log|z|}{|h(z)|\,\log|h(z)|}.$$

Since $a_1\neq 0$, we have $|h(z)|\,\log|h(z)|\sim_{z\to 0}|a_1z|\,\log|z|=|h'(0)|\,|z|\,\log|z|$, which implies that the conformal factor tends to $1$, as claimed in the lemma.
\end{proof}

Now, let us come back to the proof of Proposition \ref{sectionlipschitz}. It is enough to prove that the developing map $\DD$ is Lipschitz over a fundamental ideal polygon $P$. Such a polygon may be written as a union of a compact part and a finite number of cusps. It is possible to assume that all cusps $C_i\dans\Sigma$ lie inside a holomorphic chart where the metric reads as \eqref{standard}.

For the Fubini-Study distance, the diameter of $\C\PP^1$ is $\pi/2$. Hence it is enough to prove that $\DD$ is Lipschitz in restriction to each cusp, and to the closed $\pi/2$-neighbourhood of the compact part. The latter is immediate since this closed neighbourhood is compact and $\DD$ is holomorphic.

Now choose a cusp $C_i$ and consider a connected component $\widetilde{C}_i$ of $pr^{-1}(C_i)$, associated to the parabolic element $\gamma_i\in\pi_1(\Sigma)$. By the parabolicity of the structure, inside $\widetilde{C}_i$, $\DD$ reads as $A_i\circ h_i$, where $A_i\in PSL_2(\C)$ conjugates the actions of $z\mapsto z+1$ and $\rho(\gamma_i)$, and $h_i:\widetilde{C}_i\mapsto\Hyp_{\geq 1}$ is biholomorphic onto its image and conjugates the actions of $\gamma_i$ and $z\mapsto z+1$. It is enough to prove that $h_i$ is Lipschitz in a fundamental domain of the action of $\gamma_i$.

\begin{center}
\textbf{Claim.} \emph{Inside a fundamental domain for $\gamma_i$, $h_i$ is Lipschitz for the hyperbolic metric at the source and at the goal.}
\end{center}

Establishing the claim suffices to end the proof, because inside $\Hyp_{\geq 1}$ the hyperbolic metric is conformally equivalent the the Fubini-Study metric with a conformal factor given by:
$$\fhi(x+\mathbf{i}y)=\frac{y}{1+x^2+y^2},$$
which is smaller than $1$ in $\Hyp_{\geq 1}$.

It remains to prove the claim. First, consider the projection $\Hyp\to\D^*$ given by $z\mapsto e^{2\mathbf{i}\pi z}$: it is invariant by $z\mapsto z+1$, and the projection of the hyperbolic metric is precisely the standard metric \eqref{standard}. The biholomorphism $h_i$ passes to the quotient and gives a biholomorphism of $C_i$ inside a domain $D_i$ which is strictly included in $\D^*$ (it lies in fact in the domain $0<|z|<e^{-2\pi}$). The hyperbolic metric reads as \eqref{standard} in a holomorphic chart: we can now use Lemma \ref{hologerm}, as well as the fact that $C_i,D_i$ lie strictly inside $\D^*$. This biholomorphism, and thus $h_i$, is Lipschitz for the hyperbolic metric.

\subsection{Reduction to a problem of integrability}

\paragraph{Subexponential evolution of the distance.} In the sequel we intend to prove a stronger statement than Proposition \ref{diagonaloutsidelyap} which clearly implies it. We shall prove that the evolution of the distance between the diagonal and Lyapunov sections is subexponential along a typical orbit of the geodesic flow.
\begin{proposition}
\label{subexponential}
Let $(\DD,\rho)$ be a non-elementary parabolic branched projective structure on a hyperbolic surface of finite type $\Sigma$. Let $(\Pi,M,\Sigma,\C\PP^1,\F)$ be the associated Riccati foliation and $\sigma^0$ be the associated diagonal section. Then there exists a Borel set $\X$ which is invariant by the geodesic flow and full for the Liouville measure such that for every $v\in\X$:
$$\lim_{t\to\infty}\frac{1}{t}\log\,\dist_{\C\PP^1}(\sigma^+(g_t(v)),\sigma^0(g_t(v)))=0.$$
\end{proposition}

\paragraph{Problem of integrability.}

In order to prove Proposition \ref{subexponential}, we will use the following classical fact which is an application of Borel-Cantelli lemma.

\begin{proposition}
\label{ergodic}
Let $(X,\B,\mu)$ be a probability space, and $T:X\to X$ be a $\mu$-preserving transformation. Let $\fhi:X\to\R$ be a measurable function which is $\mu$-integrable. Then there is a set $\X\dans X$ which is $T$-invariant and full for $\mu$ such that for every $x\in\X$,
$$\lim_{n\to\infty}\frac{1}{n}\fhi\circ T^n(x)=0.$$
\end{proposition}

Proposition \ref{subexponential} is now a consequence of Proposition \ref{ergodic} and of the following property of integrability whose proof is the object of section \ref{Proofofintegrability}.

\begin{proposition}
\label{integrability}
Let $(\DD,\rho)$ be a non-elementary parabolic branched projective structure on a hyperbolic surface of finite type $\Sigma$. Let $(\Pi,M,\Sigma,\C\PP^1,\F)$ be the associated Riccati foliation and $\sigma^0$ be the associated diagonal section. Then the measurable function defined by the following formula
\begin{equation}
\label{defpsi}
\psi(v)=\Sup_{t\in[0;1]} \log\,\dist_{\C\PP^1}(\sigma^+(g_t(v)),\sigma^0(g_t(v)))
\end{equation}
 is Liouville-integrable.
\end{proposition}

\section{Proof of the integrability}
\label{Proofofintegrability}

In order to prove the integrablity of $\psi$, it is convenient to work in the cover $T^1\Hyp=\Hyp\times\R\PP^1$ endowed with coordinates that trivialize the center-unstable foliation. The main idea is to use the facts that $\sigma^+$ commutes with the center-unstable foliations, and $\sigma^0$ with the foliations by unit tangent fibers. Hence when we lift them to the cover, they realize as graphs of functions of $(z,\xi)\in \Hyp\times\R\PP^1$ in $\C\PP^1$, the first one depending only on the $\xi$ variable, and the second one depending only on the $z$ variable. Finally we are able to separate variables, which simplifies a lot the computations.

\subsection{The center-unstable foliation and lifts of the sections}

\paragraph{Trivialization of the center-unstable foliation.} We may consider the identification $T^1\Hyp\simeq \Hyp\times \R\PP^1$ obtained by sending $v$ on the couple $(c_v(0),c_v(-\infty))$ where $c_v$ is the directed geodesic determined by $v$.

This identification is an equivariance: it conjugates the actions of the group of direct isometries $PSL_2(\R)$ on $T^1\Hyp$ by differentials and on $\Hyp\times \R\PP^1$ by diagonal maps. Moreover, it also trivializes the center-unstable foliation: a slice $\Hyp\times\{\xi\}$ has to be thought as filled with unstable horocycles centered at $\xi$, and geodesics starting at $\xi$.

We denote by $G_t(z,\xi)$ the restriction of the geodesic flow to the center-unstable leaf $\Hyp\times\{\xi\}$. Hence, each of the slices $\Hyp\times\{\xi\}$ has a foliation denoted by $\G_{\xi}$, which is defined as the orbit space of this restricted geodesic flow.

In these coordinates, the Liouville measure is obtained by integration against the length element $d\xi$ of the measures:
$$dm_{\xi}(z)=k(z;\xi)d\Leb(z)=\frac{y}{(x-\xi)^2+y^2}\frac{dx\,dy}{y^2}.$$
The density $k(z;\xi)$ is the famous \emph{Poisson kernel} inside the hyperbolic plane.

\paragraph{Lifts of the sections.} The section $\sigma^+$ can be lifted as an equivariant section 
$$\widetilde{\sigma}^+:\Hyp\times\R\PP^1\to\Hyp\times\R\PP^1\times\C\PP^1.$$
Since $\sigma^+$ commutes with the center-unstable foliations, the lift reads in these coordinates as:
$$\widetilde{\sigma}^+(z,\xi)=(z,\xi,\widetilde{s}^+(\xi)),$$
where $\widetilde{s}^+:\R\PP^1\to\C\PP^1$ is a measurable map satisfying the equivariance relation $\widetilde{s}^+\circ\gamma=\rho(\gamma)\circ\widetilde{s}^+$ for every $\gamma\in\pi_1(\Sigma)$.

Similarly in these coordinates the lift of $\sigma^0$, which is the developing map, reads as follows:
$$\DD(z,\xi)=(z,\xi,\DD(z)).$$

We shall fix now a fundamental ideal polygon $P\dans\Hyp$ which can be decomposed as a union of a compact part $K$, and of $2l$ cusps $C^+_i$ bounded by the geodesic sides of the polygon, as well as by segments of horocycles where we recall that $l$ is the maximal number of mutually disjoint and non-homotopic geodesics whose ends arrive to punctures. Proving Proposition \ref{integrability} is equivalent to proving that:
$$I=\iint_{P\times\R}\widetilde{\psi}(z,\xi)\,dm_{\xi}(z) d\xi<\infty,$$
where:
\begin{equation}
\label{lafonctionpsi}
\widetilde{\psi}(z,\xi)=\Sup_{t\in[0;1]}\left|\log \dist_{\C\PP^1}(\widetilde{s}^+(\xi),\DD(G_{t}(z,\xi)))\right|.
\end{equation}

We will decompose this integral as a sum $I_K+\sum_{j=1}^{2l} I_{C^+_j}$, where $I_K$ is the above integral taken on $K\times\R$, and $I_{C_j^+}$ on $C^+_j\times\R$, and prove that each of these terms are finite.

\subsection{Integrability over the compact part}
\label{Compactpart}

\paragraph{Foliations of the compact part.} For $\xi\in\R$, consider the set $K_{\xi}=\bigcup_{z\in K}G_{[0;1]}(z,\xi)$, and $K'=\bigcup_{\xi\in\R}K_{\xi}$. The set $K'$ is a compact subset of $\Hyp$, and has the property that for all $z\in K$ and $\xi\in\R$, $G_{[0;1]}(\xi,z)\dans K'$.

It is also foliated by the traces of $\G_{\xi}$: when $z\in\Hyp$, let $L_{\xi}(z)$ denote $K'\cap\G_{\xi}(z)$. Then we have that for all $z\in K$ and $\xi\in\R$, $\widetilde{\psi}(z,\xi)\leq \widetilde{\psi}'(z,\xi)$ where:

$$\widetilde{\psi}'(z,\xi)=\Sup_{w\in L_{\xi}(z)}\left|\log \dist_{\C\PP^1}(\widetilde{s}^+(\xi),\DD(w))\right|.$$
Hence, in order to deal with the compact part, it is enough to prove the integrability over $K'\times\R$ of $\widetilde{\psi}'$.

Remark that the latter function is constant along the $L_{\xi}(z)$: it will be useful for the proof.

\paragraph{Decomposition of the compact part.} The developing map $\DD$ is holomorphic and nonconstant. As a consequence it has only a finite number of critical points in the compact set $K'$, that we denote by $(a_j)_{j\in J}$.

Hence there exist a number $\delta>0$, a finite number of \emph{disjoint} discs $(U_j)_{j\in J}$ centered at $a_j$ and of hyperbolic radii $2\delta$, as well as a finite number of discs $(V_{\alpha})_{\alpha\in A}$ of hyperbolic radii $\delta$ such that:

\begin{itemize}
\item $K'\dans\bigcup_{j\in J}U_j\cup\bigcup_{\alpha\in A}V_{\alpha}$;
\item for every $j\in J$ and $\alpha\in A$, $V_{\alpha}\cap D_{\Hyp}(a_j,\delta)=\vide$;
\item for every $j\in J$ and $\alpha\in A$, $\DD(U_j)$ and $\DD(V_{\alpha})$ are proper open sets of $\C\PP^1$;
\item when restricted to $V_{\alpha}$, the developing map is a biholomorphism to its image.
\end{itemize}

Since $\DD(U_j)$ are proper open sets of $\C\PP^1$, each of these sets are included in an affine chart so that we can imagine these sets as included in $\C$. Hence the restriction $\DD_{|U_j}$ reads as follows: there exist an integer $n_j>1$ and a map $h_j:U_j\to\C$ which is a biholomorphism on its image such that for any $z\in U_j$:

\begin{equation}
\label{power}
\DD(z)-\DD(a_j)=h_j(z)^{n_j}.
\end{equation}

\paragraph{Lower bound for the distance between the two sections.} The following lemma allows us to treat the problem of existence of critical points.

\begin{lemma}
\label{lowerbounddev}
There exists a constant $C_0>0$ such that for all $z\in K'$ and $s\in\DD(K')$, we have:
$$\dist_{\C\PP^1}(\DD(z),s)\geq C_0 \prod_{\DD(w)=s}\dist_{\Hyp}(z,w).$$
\end{lemma}

\begin{proof}
We start by noticing that by the decomposition of $K'$ that we described in the previous paragraph, there exists an integer $n$ such that each element of $\DD(K')$ has at most $n$ preimages in $K'$. Hence, since $K'$ has finite diameter, it is enough to find a constant $C>0$ such that for $z\in \Omega$, $\Omega$ being either a set $U_j$ or a set $V_{\alpha}$, and $s\in\DD(\Omega)$:

\begin{equation}
\label{localinequality}
\dist_{\C\PP^1}(\DD(z),s)\geq C \prod_{w\in\Omega,\DD(w)=s}\dist_{\Hyp}(z,w).
\end{equation}

When $\Omega$ is of the form $V_{\alpha}$, Inequality \eqref{localinequality} holds for some universal $C$ because $\DD$ is a biholomorphism in restriction to each of these open sets, which are uniformly far from the critical points: the derivatives of the $\DD_{|V_{\alpha}}$ are uniformly bounded away from zero.

When $\Omega$ is of the form $U_j$, it contains a unique critical point $a_j$. Then as we mentioned above, there exist an integer $n_j>1$ and a map $h_j:U_j\to\C$ such that in an affine chart the restriction of $\DD$ to $U_j$ reads as (\ref{power}). By compactness, in restriction to $U_j$ and to $\DD(U_j)$ we may compare respectively the hyperbolic and spherical distances with the euclidian one with a uniform distortion. Then Inequality \eqref{localinequality} will hold with the euclidian distance because $h_j$ is a biholomorphism (its derivative is bounded away from zero independently of $j$) and because for every $z_1,z_2\in\C$ and $n\in\N$, we have the following equality:

$$|z_1^n-z_2^n|=\prod_{w^n=z_2^n}|z_1-w|.$$
\end{proof}

\paragraph{Upper bound of the integral.} Over the compact $K'$, the Poisson kernel $k(z,\xi)$ is, up to a uniform multiplicative constant, controlled by $1/(1+\xi^2)$, which is integrable over $\R$. Hence, by Fubini, it is enough to find a constant $C_1$ independent of $\xi\in\R$ such that for every $z\in K'$:
$$\int_{K'}\widetilde{\psi}'(z,\xi)d\Leb(z)\leq C_1.$$

Let $\xi\in\R$. Then there are two cases. Either $\widetilde{s}(\xi)$ belongs to the $1/1000$-neighbourhood of $\DD(K')$, or it does not. In the latter case, the function $\psi(.,\xi)$ is bounded from above by $\log(1000)$ in $K'$, which has a finite area. In the first case we can, by pushing it slightly by the geodesic flow, enlarge the compact $K'$ in such a way that $\DD(K')$ contains $\widetilde{s}(\xi)$.

Assume now that $\DD(K')$ contains $\widetilde{s}^+(\xi)$. The number of preimages of $\widetilde{s}^+(\xi)$ inside $K'$ is finite and bounded independently of $\xi\in\R$. By Lemma \ref{lowerbounddev}, and since $\C\PP^1$ has finite diameter, we find that there is a constant $C>0$ which is independent of $\xi$ such that:
$$\widetilde{\psi}'(z,\xi) \leq C+\sum_{\DD(\zeta)=\widetilde{s}^+(\xi)}\widetilde{\phi}(z,\zeta,\xi),$$
where $\widetilde{\phi}(z,\zeta,\xi)=\Sup_{w\in L_{\xi}(z)}|\log\dist_{\Hyp}(\zeta,w)|$. It is then enough to bound from above by a uniform constant the integral over $K'$ of each $\widetilde{\phi}(.,\zeta,\xi)$, $\zeta\in K'$, $\xi\in\R$.

\paragraph{Integrability over the compact part: end of the proof.} We have shown that in order to get the integrability of $\widetilde{\psi}$ it is enough to prove the following lemma:

\begin{lemma}
\label{compactcase}
The integrals over $K'$ of the functions $\widetilde{\phi}(.,\zeta,\xi)$, $\zeta\in K'$, $\xi\in\R$ against the Lebesgue measure are bounded independently of $\zeta,\xi$.
\end{lemma}

\begin{proof}
The set $K'$ is compact and foliated by the $L_{\xi}(z)$. Hence, passing through each point $\zeta\in\ K'$, there is a segment of horocycle  centered at $\xi$, denoted by $H_{\xi}(\zeta)$, whose length is bounded independently of $\zeta$ and $\xi$ and such that $K'\dans\bigcup_{z\in H_{\xi}(\zeta)}L_{\xi}(z)$.

In the \emph{compact} set $K'$ we have the following facts.

\begin{enumerate}
\item It is well known that given two points $z_1$, $z_2$ that belong to the same horocycle, the horocyclic distance between them, which we denote by $\dist_{horo}(z_1,z_2)$, is given by the following formula:

\begin{equation}
\label{comphorogeo}
\dist_{horo}(z_1,z_2)=2\sinh\frac{\dist_{\Hyp}(z_1,z_2)}{2}.
\end{equation}
Hence for $z,\zeta\in K'$ on the same horocycle centered at $\xi$, the horocyclic and geodesic distances between them are in a uniformly log-bounded ratio.
\item Since the horocycle segments $H_{\xi}(\zeta)$ have bounded lengths and curvatures, their arc length parametrizations are uniformly bounded independently of $\xi$ in the $C^1$-norm.
\item Since moreover each horocycle $H_{\xi}(\zeta)$ is orthogonal to the geodesic segments $L_{\xi}(\zeta)$, we obtain by Fubini that the Lebesgue measure is, when restricted to $K'$, equivalent to the measure obtained by integration of the arc element along the geodesics $L_{\xi}(z)$ against the arc length element along the horocycle $H_{\xi}(\zeta)$, with a Radon-Nikodym derivative which is log-bounded independently of $\zeta,\xi$.
\item The functions $\widetilde{\phi}(.,\zeta,\xi)$ are constant along the $L_{\xi}(z)$ whose lengths are uniformly bounded.
\end{enumerate}

From all this we find a number $C>0$ such that for all $\zeta\in K'$, $\xi\in\R$:
$$\int_{K'}\widetilde{\phi}(z,\zeta,\xi)d\Leb(z)\leq C\int_{H_{\xi}(\zeta)} |\log\,\dist_{horo}(\zeta,s)|\,d\lambda_{\xi,\zeta}(s),$$
where $\lambda_{\xi,\zeta}$ denotes the arc length element of $H_{\xi}(\zeta)$.

Now since the logarithm is integrable at $0$, and since the arc length parametrizations of the horocycles $H_{\xi}(\zeta)$ are $C^1$-uniformly bounded, a change of variable shows that these curve integrals are uniformly bounded.  This finishes the proof of the lemma.
\end{proof}

\subsection{Integrability over the cusps: the model case}

\paragraph{The inclusion.} Recall that a local model for the developing map of a parabolic structure in a cusp is given by the inclusion $\iota:C^+=\left[-\frac{1}{2};\frac{1}{2}\right]\times[1;\infty)\hookrightarrow\C\PP^1$. In order to study the integrability problem over a cusp, we will first treat the analogous problem for this model. In the final paragraph, we will perform a change of coordinate, in order to treat the general case.

In coordinates, the Liouville measure reads in $\Hyp$ as:
$$d\Liouv=k(x+\mathbf{i}y;\xi)\frac{dx\,dy}{y^2}\,d\xi.$$

The goal of this paragraph is to prove that  when $\DD=\iota$, the following integral is finite:
$$\int_{\R}\int_1^{\infty}\frac{1}{y^2}\int_{-\frac{1}{2}}^{\frac{1}{2}}\widetilde{\psi}(x+\mathbf{i}y,\xi) k(x+\mathbf{i}y;\xi) dx\,dy\,d\xi.$$

\paragraph{Remark 1.} Since we know how to prove the integrability of our function inside the compact part, we can enlarge it if necessary and in particular it is enough to study this problem for $y$ lying in the interval $[y_0;\infty)$ for some $y_0\gg 1$: we shall choose this constant later. 

\paragraph{Decomposition of the integral.} We will use bounds on the Poisson kernel in order to decompose the integral into two parts. Recall that the following formula holds for all $x,\xi\in\R$ and $y>0$:
$$k(x+\mathbf{i}y;\xi)=\frac{y}{(x-\xi)^2+y^2}.$$
Choose $y_0$ large enough so that there is a uniform constant $C>1$ such that for every $x\in [-1/2;1/2]$ and $y\in[y_0;\infty)$:
\begin{enumerate}
\item when $|\xi|\geq -1/2$ then:
$$k(x+\mathbf{i}y;\xi)\leq C\frac{y}{\xi^2+y^2};$$
\item when $|\xi|< 1/2$ then:
$$k(x+\mathbf{i}y;\xi)\leq 1.$$
\end{enumerate}

As a consequence, it is enough to prove that the three following integrals are finite:

\begin{eqnarray*}
I_{[-1/2;1/2]}    &=&\int_{-\frac{1}{2}}^{\frac{1}{2}}\int_{y_0}^{\infty}\frac{1}{y^2}\int_{-\frac{1}{2}}^{\frac{1}{2}} \widetilde{\psi}(x+\mathbf{i}y,\xi)  dx\,dy\,d\xi  \\
I_{(-\infty;-1/2]}&=&\int_{-\infty}^{-\frac{1}{2}}\int_{y_0}^{\infty}\frac{1}{y}\frac{1}{\xi^2+y^2}\int_{-\frac{1}{2}}^{\frac{1}{2}} \widetilde{\psi}(x+\mathbf{i}y,\xi)  dx\,dy\,d\xi\\
I_{[1/2;\infty)}  &=&\int_{\frac{1}{2}}^{\infty}\int_{y_0}^{\infty}\frac{1}{y}\frac{1}{\xi^2+y^2}\int_{-\frac{1}{2}}^{\frac{1}{2}} \widetilde{\psi}(x+\mathbf{i}y,\xi)  dx\,dy\,d\xi 
\end{eqnarray*}

We will first have to use a geometric argument in order to bound the integrals of the function $\widetilde{\psi}$ on horizontal slices $[-1/2;1/2]\times\{y\}$, and then conclude by simple calculus.

\paragraph{Pencils of geodesics.} Given real numbers $\xi\in\R$, and $y\geq 1$, we can consider the pencil of geodesics starting at $\xi$ and passing through the horizontal slice $[-1/2;1/2]\times\{y\}$. Denote this pencil by $\G_{\xi}$. Denote by $\G_{\xi}(z)$ the geodesic passing through $\xi$ and $z$. Denote by $L_{\xi}(z)$ the orbit segment $G_{[0;1]}(z,\xi)$. We want to estimate the distance between $\widetilde{s}^+(\xi)$ and the segments $L_{\xi}(x+\mathbf{i}y)$, $x\in[-1/2;1/2]$ and $y\geq 1$. We will be interested in the part of this pencil defined by:
\begin{equation}
\label{axi}
A_{\xi}(y)=\bigcup_{x\in[-1/2;1/2]}L_{\xi}(x+\mathbf{i}y).
\end{equation}

\paragraph{Remark 2.} In these coordinates the function $\widetilde{\psi}$ reads as follows for $z=x+\mathbf{i}y$ with $x\in[-1/2;1/2]$ and $y\geq 1$:
$$\widetilde{\psi}(z,\xi)=|\log\,\dist_{\C\PP^1}(\widetilde{s}^+(\xi),L_{\xi}(z))|.$$

\paragraph{Remark 3.} The set $\lbrace(x+\mathbf{i}y,\xi);x=\xi, y\geq 1\rbrace$ is of zero Liouville measure. Hence in the sequel, we will only consider the case where $x\neq\xi$. In other words, the only (hyperbolic) geodesics we shall consider are euclidian half circles orthogonal to the horizontal axis.

\begin{lemma}
\label{square}
There exists a uniform constant $C_0>1$ such that for $y\geq 1$ and $\xi\in\R$, the set $A_{\xi}(y)$ is included in the part of the complex plane identified with $[-C_0 y;C_0 y]\times[C_0^{-1}y;C_0 y]$.
\end{lemma}

\begin{proof}
It is enough to prove the existence of a constant $C>0$ such that along every geodesic ray $L$ of length $1$, the variation of real and imaginary parts is $\leq C y$, where $y$ is the lowest imaginary part of a point of $L$.

By applying a similitude of the complex plane, which is a hyperbolic isometry, it is enough to assume that the geodesic segment $L$ is included in the half circle centered at $0$ and of euclidian radius $1$. More precisely, we intend to prove that the real and imaginary parts of elements of $G_{[0;1]}(e^{\mathbf{i}\theta},1)$ vary in intervals uniformly of the order of $\sin\theta=\im(e^{\mathbf{i}\theta})$. By symmetry, it is enough to consider the case where $\theta\in(0;\pi/2]$.

First, notice that when $\theta_1<\theta_2$:
$$\dist_{\Hyp}(e^{\mathbf{i}\theta_1},e^{\mathbf{i}\theta_2})=\int_{\theta_1}^{\theta_2}\frac{d\theta}{\sin\theta}=\log\frac{\tan(\theta_2/2)}{\tan(\theta_1/2)}.$$

This implies that if $k(\theta)>1$ is defined in such a way that $\dist_{\Hyp}(e^{\mathbf{i}\theta},e^{\mathbf{i}k(\theta)\theta})=1$, then $k(\theta)$ has to be uniformly bounded from above. In order to see this, use $\tan\theta\sim\theta$ for $\theta$ small, as well as a lower bound $c$ of the derivative $|D(\log\circ\tan)|$, which is uniform in some compact interval $[\eps;\pi/4]$. Using the equality $1=\log [\tan(k(\theta)\theta/2)/\tan(\theta/2)]$ we obtain, for $\eps$ small enough, that $k(\theta)\leq 2e$ for $\theta\in(0;2\eps]$, and $k(\theta)\leq 1+2(c\eps)^{-1}$ for $\theta\in[2\eps;\pi/2]$.

In other words, along a geodesic segment of length $1$ starting at $e^{\mathbf{i}\theta}$ the argument stays uniformly of the order of $\theta$.

By the Lipschitz property, when $s\in[1,k(\theta)]$, we have that $|\cos(s\theta)-\cos\theta|$ and $|\sin(s\theta)-\sin\theta|$ are smaller than $(k(\theta)-1)\theta$. Since moreover $\theta \leq \pi/2\sin\theta$ in $[0;\pi/2]$, we obtain the desired uniform bound.
\end{proof}

Hence it allows us to work in $A_{\xi}(y)$ with euclidian, spherical or hyperbolic metrics indinstincly with a controlled distortion.

\begin{lemma}
\label{conformaldistortion}
There is constant $C_1<1$ such that for every $y\in[1;\infty)$ and $\xi\in\R$, we have:
$$\dist_{\C\PP^1}(z_1,z_2)\geq \frac{C_1}{y}\dist_{\Hyp}(z_1,z_2)\geq \frac{C_1^2}{y^2}|z_1-z_2|,$$
when $z_1,z_2\in A_{\xi}(y)$.
\end{lemma}

\begin{proof}
The spherical and hyperbolic metrics are conformally equivalent with respect to the euclidian one with conformal factors respectively given by $1/(1+x^2+y^2)$ and $1/y$. A use of Lemma \ref{square} allows us to conclude.
\end{proof}

We need a third lemma which allows us to compare horocyclic and geodesic distance in hyperbolic geometry.

\begin{lemma}
\label{horogeo}
There exists a number $y_0\gg 1$ such that for every $y\geq y_0$, $\xi\in\R$ and $x_1,x_2\in[-1/2;1/2]$ we have:
$$\dist_{\Hyp}(z_1,L_{\xi}(z_2))\leq\dist_{horo}(z_1,L_{\xi}(z_2))\leq 2 \dist_{\Hyp}(z_1,L_{\xi}(z_2)),$$
where $z_k=x_k+\mathbf{i} y$.
\end{lemma}

\begin{proof}
Notice that when $y\geq y_0$ and $x_1,x_2\in[-1/2;1/2]$ we have uniformly $\dist_{\Hyp}(x_1+\mathbf{i}y,x_2+\mathbf{i}y)\leq 1/y_0$. Hence using Formula (\ref{comphorogeo}) as well as a uniform Lipschitz constant of $\sinh$ in a neighbourhood of $0$ we conclude that the inequality holds when $y_0$ is large enough.
\end{proof}

\paragraph{Integrals on the horizontal slices.} The following proposition is the main technical ingredient: it will allow us to conclude the proof by simple calculus.

\begin{proposition}
\label{horizontal}
There exists constants $y_0>1$ and $C>0$ such that for every $y\in[y_0;\infty)$ and $\xi\in\R$:
$$\int_{-\frac{1}{2}}^{\frac{1}{2}} \widetilde{\psi}(x+\mathbf{i}y,\xi)  dx \leq C \log(\xi^2+y^2).$$
\end{proposition}

The function $\widetilde{\psi}$ has been defined as the log of the Fubini-Study distance of $\widetilde{s}^+$ to geodesic segments $L_{\xi}(z)$. The idea of the proof is to control the restriction of this function to horizontal slices, up to logarithmic quantities, by the log of the euclidian distance of the projection of $\widetilde{s}^+(\xi)$ on this horizontal slice. Using the integrability of the logarithm in the neighbourhood of $0$ we will be able conclude the proof.

Before we carry on the proof let us make the following comment: \emph{in order to prove the proposition it is enough to assume that} $\widetilde{s}^+(\xi)\in A_{\xi}(y)$. Indeed we can again distinguish two cases. Either it lies at distance $\geq 1/1000$ of $A_{\xi}(y)$ and $\widetilde{\psi}(x+\mathbf{i}y,\xi)\leq\log(1000)$ so that the estimation stated in Proposition \ref{horizontal} is valid. Or it belongs to the $1/1000$-neighbourhood of $A_{\xi}(y)$ and, by slightly enlarging the interval $[-1/2;1/2]$, we come down to the case $\widetilde{s}^+(\xi)\in A_{\xi}(y)$.

\paragraph{Auxiliary functions.} We will prove Proposition \ref{horizontal} by coming down to a problem of euclidian geometry. In order to do this we need to consider four auxiliary functions.

Assuming $\widetilde{s}^+(\xi)\in A_{\xi}(y)$ for some $y\geq y_0$ and $\xi\in\R$ we consider $s_0$ the projection of $\widetilde{s}^+(\xi)$ on $[-1/2;1/2]\times\{y\}$ along $L_{\xi}(\widetilde{s}^+(\xi))$. Of course there is the possibility that the geodesic segment $L_{\xi}(\widetilde{s}^+(\xi))$ meets $[-1/2;1/2]\times\{y\}$ twice in which case the projection is not well defined. If it occurs, we define $s_0$ as the intersection with the least real part when $\xi<0$, and with greatest real part when $\xi>0$.

We define for $z=x+\mathbf{i}y$, $x\in[-1/2;1/2]$:
\begin{itemize}
\item $\widetilde{\psi}_1(z,\xi)=|\log\,\dist_{\Hyp}(\widetilde{s}^+(\xi),L_{\xi}(z))|$;
\item $\widetilde{\psi}_2(z,\xi)=|\log\,\dist_{\Hyp}(s_0,L_{\xi}(z))|$;
\item $\widetilde{\psi}_3(z,\xi)=|\log\,\dist_{\C}(s_0,L_{\xi}(z))|$;
\item $\widetilde{\psi}_4(z,\xi)=|\log\,\dist_{\C}(s_0,\G_{\xi}(z))|$.
\end{itemize}

Denote by $J_k$ the integral $\int_{-1/2}^{1/2}\widetilde{\psi}_k(x+\mathbf{i}y,\xi)dx$ for $k=1,2,3,4$.

\begin{lemma}
\label{reduceeuclidian}
Let $y\geq y_0$ and $\xi\in\R$ with $\widetilde{s}^+(\xi)\in A_{\xi}(y)$. Then there exist positive constants $C_1,C_2,C_3,C_4$ independent of $\xi$ such that for all $x\in[-1/2;1/2]$:
\begin{eqnarray}
\label{ineq1}
\int_{-1/2}^{1/2}\widetilde{\psi}(x+\mathbf{i}y,\xi)\,dx &\leq& J_1+C_1\log y\\
\label{ineq2}
                                                    &\leq& J_2+C_2\log y\\
\label{ineq3}
                                                    &\leq& J_3+C_3\log y\\
\label{ineq4}
                                                    &\leq& J_4+C_4\log y,
\end{eqnarray}
\end{lemma}

\begin{proof}
Inequality (\ref{ineq1}) follows directly from Lemma \ref{conformaldistortion} where we compare the Fubini-Study and hyperbolic distances inside $A_{\xi}(y)$.

Inequality (\ref{ineq2}) follows from Lemma \ref{horogeo} where it is proven that horocyclic and geodesic distances are comparable in $A_{\xi}(y)$ when $y\geq y_0$. Indeed, for $z=x+\mathbf{i}y$, $x\in[-1/2;1/2]$  the horocyclic projection of $\widetilde{s}^+(\xi)$ (resp. $s_0$) on the geodesic segment $L_{\xi}(z)$ is defined by sliding along the horocyle centered at $\xi$ and passing through $\widetilde{s}^+(\xi)$ (resp. $s_0$) which is both orthogonal to $L_{\xi}(z)$ and $L_{\xi}(\widetilde{s}^+(\xi))=L_{\xi}(s_0)$. Since the geodesic segment $[s_0;\widetilde{s}^+(\xi)]$ has a length bounded by $1$ it means that these two horoyclic distances are in a uniformly bounded ratio, thus proving Inequality (\ref{ineq2}).

Inequality (\ref{ineq3}) also follows from Lemma \ref{conformaldistortion} where we compare the hyperbolic and euclidian distances inside $A_{\xi}(y)$.

The last inequality is trivial.
\end{proof}

\paragraph{Euclidian geometry.} Lemma \ref{reduceeuclidian} enables us to deal with euclidian orthogonal projections on hyperbolic geodesics which we recall are euclidian half circles. It is very easy to compute euclidian radii of the geodesics $\G_{\xi}(x+\mathbf{i}y)$ and to see that in particular they are uniformly bounded from below independently of $\xi,x\in[-1/2;1/2]$ and $y\geq y_0$.

\begin{lemma}
\label{euclidianradius}
Let $x\in[-1/2;1/2]$, $y\geq y_0$ and $\xi\in\R$. Then the euclidian radius of the geodesic $\G_{\xi}(x+\mathbf{i}y)$ is given by
\begin{equation}
\label{radius}
R_x=\frac{(\xi-x)^2+y^2}{2|\xi-x|}.
\end{equation}
In particular we always have $R_x\geq y_0$.
\end{lemma}

\begin{proof}
We will prove it for $x=0$ and $\xi>0$. Using Pythagoras' theorem in the triangle whose vertices are $0$, $\mathbf{i}y$ and the euclidian center of the geodesic gives the following relation:
$$R^2=(R-\xi)^2+y^2,$$
from which first assertion of the lemma follows easily.

In order to see that we always have $R\geq y_0$ we apply the inequality of arithmetic and geometric means to $y^2$ and $\xi^2$ and remember that $y\geq y_0$.
\end{proof}

We will use the following lemma which compares the orthogonal projection of a point on the unit circle with its projection in the horizontal direction. Consider the first quadrant $S^+=\lbrace z\in\C;|z|=1,\,\re(z)\geq 0,\im(z)\geq 0\rbrace$, and for a small $\eps$, consider $S^+(\eps)=\lbrace z\in S^+;\re(z)\geq\eps\rbrace$. Consider also the images $S^-$ and $S^-(\eps)$ of these sets by the reflection $z\mapsto -\bar{z}$.

\begin{lemma}
\label{lagrange}
If $\eps$ is small enough we have:
\begin{enumerate}
\item for every $z\in S^{\pm}$:
$$\dist_{\C}(z\pm\eps,S^1)\geq \eps^2.$$
\item for every $z\in S^{\pm}(\eps)$:
$$\dist_{\C}(z\mp\eps,S^1)\geq \eps^2.$$
\end{enumerate}
\end{lemma}

\begin{proof}
Let us show the first assertion. By symmetry it is enough to prove the statement when $z\in S^+$.

Consider the function of the complex variable $f(z)=\dist_{\C}(z+\eps,S^1)$ as well as the constraint function $g(z)=|z|^2-1$. The function $f$ is smooth on $S^+$ (since it does not vanish) and $g$ is smooth everywhere.

By the theory of Lagrange multipliers, if an interior point of the arc $S^+$ is a local extremum of $f$ then the gradients $\nabla f$ and $\nabla g$ are colinear at this point. But for every $z\in S^+$, $\nabla_z g$ is colinear to the vector $z$ and $\nabla_zf$ is colinear to the vector $z+\eps$. In other words if $z$ is a local extremum then $z$ and $z+\eps$ are colinear: this is only possible if $z=1$. Finally we find that there is no local extremum in the interior of $S^+$.

Hence the extrema of the restriction of $f$ to the arc $S^+$ are precisely its extremities. But by Pythagoras' theorem $f(\mathbf{i})=\sqrt{1+\eps^2}-1\sim\eps^2/2<\eps=f(1)$ for $\eps$ small enough. Hence when $\eps$ is small enough one has for all $z\in S^+$, $f(z)\geq f(\mathbf{i})\geq \eps^2$.

The second assertion follows by the same type of arguments. Indeed, because when $z\in S^+(\eps)$, $z-\eps\notin S^1$, we have that the function $h(z)=\dist_{\C}(z-\eps,S^1)$ is smooth on $S^+(\eps)$, and the argument of Lagrange multipliers is again valid.
\end{proof}

\paragraph{Proof of Proposition \ref{horizontal}: case 1.} Call \emph{North pole} of a circle of $\C$ its point with highest imaginary part. Choose $y\geq y_0$ and $\xi\in\R$. For $x\in[-1/2;1/2]$, we shall denote by $N(x)$ the North pole of $\G_{\xi}(x+\mathbf{i}y)$. The first case we treat is the following.

\begin{center}
$(\ast)\,\,\,\,$ \emph{All North poles of geodesics starting at $\xi$ and passing through} $(-1/2;1/2)\times\{y\}$ \emph{have real part outside of} $[-1/2;1/2]$.
\end{center}

In particular, in that case every geodesic of the pencil $\G_{\xi}$ passing through the horizontal slice $[-1/2;1/2]\times\{y\}$ intersects it only once. Note that, at least when $y_0$ is large, this case includes the case $\xi\in[-1/2;1/2]$.

Let $z=x+\mathbf{i}y$. Recall that $s_0$ is the projection on $[-1/2;1/2]\times\{y\}$ of $\widetilde{s}^+(\xi)$ along $L_{\xi}(\widetilde{s}^+(\xi))$. By Lemma \ref{euclidianradius} the euclidian radius of the geodesic $\G_{\xi}(z)$ is given by:
$$R_x=\frac{(\xi-x)^2+y^2}{2|\xi-x|}.$$

Choose $y_0$ large enough so that Lemma \ref{lagrange} is valid with $\eps<1/y_0$. Call $S$ the half circle of euclidian radius $1$ obtained by dividing $\G_{\xi}(z)$ by $R_x$.

Recall the definition of $S^+(\eps)$, with the obvious generalization to circles which are not centered at the origin. A point $z\in S$ lies in $S^+(\eps)$ if the real part of $z-\eps$ is more than or equal to that of the North Pole. We have defined $S^-(\eps)$ by applying a reflection with respect to a vertical axis.

 By the hypothesis we made on the North poles and since the segment with extremities $z$ and $s_0$ is included in $[-1/2;1/2]\times\{y\}$, we have $z/R_x\in S^{\pm}(|s_0-z|/R_x)$ (the sign depending of the relative position of $z$ and the North pole $N(x)$). Hence by use of Lemma \ref{lagrange} with $\eps=|s_0-z|/R_x$:
$$\dist_{\C}\left(\frac{s_0}{R_x},S\right)=\dist_{\C}\left(\frac{z}{R_x}+\frac{s_0-z}{R_x},S\right)\geq \left(\frac{|s_0-z|}{R_x}\right)^2.$$
Multiplying by $R_x$ gives  
$$1\geq|s_0-z|\geq\dist_{\C}(s_0,\G_{\xi}(z))\geq  \frac{|s_0-z|^2}{R_x}.$$
Passing to the logarithm and using the fact that $(\xi-x)^2+y^2$ is in uniformly log-bounded ratio with $\xi^2+y^2$ gives a constant $C>0$ such that:
$$|\log\,\dist_{\C}(s_0,\G_{\xi}(z))|\leq C+2\left|\log|s_0-z|\right|-\log^-|\xi-x|+\log(\xi^2+y^2),$$
where $\log^-(\xi)=\Min(\log\xi,0)$ is the negative part of the logarithm.

Now integrate this inequality against the variable $x\in[-1/2;1/2]$. On the one hand we have that $\int_{-1/2}^{1/2}\left|\log|s_0-(x+\mathbf{i}y)|\right|dx$ is bounded independently of $s_0$. On the other hand $\log^-|\xi-x|\neq 0$ for some $x\in[-1/2;1/2]$ only if $\xi\in[-3/2;3/2]$. But the integral of the logarithm on an interval of length $1$ inside $[-3/2;3/2]$ is uniformly bounded.

Finally we find a constant $C'>0$  such that for all $y\geq y_0$ and $\xi\in\R$ such that Hypothesis $(\ast)$ holds, we have:
$$J_4=\int_{-1/2}^{1/2}|\log\,\dist_{\C}(s_0,\G_{\xi}(x+\mathbf{i}y))|\,dx\leq C'+\log(\xi^2+y^2).$$

By Lemma \ref{reduceeuclidian} we can conclude the proof of Proposition \ref{horizontal} in the first case.\quad \hfill $\square$

\paragraph{Proof of Proposition \ref{horizontal}: case 2.} It remains to treat the following case:

\begin{center}
$(\ast\ast)\,\,\,\,$\emph{There exists} $x_1\in[-1/2;1/2]$ \emph{such that} $\re(N(x_1))\in[-1/2;1/2]$.
\end{center}

By symmetry, it is enough to treat the case $\xi>1/2$, (the case $|\xi|\leq 1/2$ has already been treated in the previous paragraph). Before we show how to deal with Hypothesis $(\ast\ast)$, let us assume the following hypothesis, which is more restrictive.

\begin{center}
$(\ast\ast')\,\,\,\,$\emph{Every  geodesic} $\G_{\xi}(x+\mathbf{i} y)$, $x\in [-1/2;1/2]$ \emph{intersects the horizontal slice} $[-1/2;1/2]\times\{y\}$ \emph{exactly twice, except one which is tangent to the slice}.
\end{center}

In that case, we have an involution $\rho:[-1/2;1/2]\to[-1/2;1/2]$ which associates to $x$ the real part of the other intersection with $[-1/2;1/2]\times\{y\}$ of the geodesic $\G_{\xi}(x+\mathbf{i} y)$. Call $x_0$ the fixed point of this involution: $\G_{\xi}(x_0+\mathbf{i} y)$ is tangent to $[-1/2;1/2]\times\{y\}$.

\begin{lemma}
Assume that Hypothesis $(\ast\ast')$ holds. The involution $\rho$ is smooth and its derivative is bounded independently of $x\in [-1/2;1/2]$, $y\geq y_0$ and $\xi$.
\end{lemma}

\begin{proof}
Consider the reflection $\rho_0:x\in[-1/2;1/2]\mapsto -x$. For $x\in[-1/2;1/2]$, consider the map $\tau_x$ defined as the translation $z\mapsto z+(R_x-\xi)$. Since $\xi> 1/2$, it is easily seen that if for each $x$ we apply $\tau_x$ to $\G_{\xi}(x+\mathbf{i}y)$ we obtain a family of concentric circles whose common center is the origin.

Hence one shows $\rho(x)=\tau_x^{-1}\circ\rho_0\circ\tau_x(x)$. Finally $\rho$ is smooth and in order to bound its derivative, it is enough to bound that of $x\mapsto R_x$.

Yet it is obvious from \eqref{radius} that $\frac{dR_x}{dx}=3/2+\frac{y^2}{2(\xi-x)^2}$ (use here that $\xi>1/2\geq x$). Since by definition of $x_0$, we have $y=\xi-x_0\geq y_0$ which is large enough, and $|x-x_0|\leq 1$, we obtain that $\xi$ is far from $1/2$ so that $\left(\frac{\xi-x_0}{\xi-x}\right)^2$, and hence the derivative of $R_x$, is clearly bounded independently of $x,\xi$ and $y\geq y_0$. This concludes the proof of the lemma.
\end{proof}

Now, assuming Hypothesis $(\ast\ast')$, we can cut the integral below into two parts, and then perform a change of variable $x'=\rho(x)$ to one of the pieces. Using the lemma above as well as the fact that $\dist_{\C}(s_0,\G_{\xi}(x+\mathbf{i}y))=\dist_{\C}(s_0,\G_{\xi}(\rho(x)+\mathbf{i}y))$ we find a constant $C>0$ such that:
$$\int_{-1/2}^{1/2}|\log\,\dist_{\C}(s_0,\G_{\xi}(x+\mathbf{i}y))|dx\leq C\int_{x_0}^{1/2}|\log\,\dist_{\C}(s_0,\G_{\xi}(x+\mathbf{i}y))|dx.$$

Now, since $\xi>1/2$, the North poles $N(x)$, $x\in[-1/2;1/2]$ have real parts outside $(x_0;1/2]$ and we can end the proof under this hypothesis as we did under Hypothesis $(\ast)$.\quad \hfill $\square$\\

Now we show that assuming Hypothesis $(\ast\ast)$, it is possible to come down to Hypothesis $(\ast\ast')$ by enlarging the horizontal slice $[-1/2;1/2]\times\{y\}$. This is object of the following:

\begin{lemma}
Assume that Hypothesis $(\ast\ast)$ holds for some $x_1\in[1/2;1/2]$, $y\geq y_0$ and $\xi\in\R$. Then, there exists an interval $I=I(y)$ containing $[-1/2;1/2]$, with length bounded independently of $y$ and $\xi$, such that every geodesic $\G_{\xi}(x+\mathbf{i}y)$, $x\in[-1/2;1/2]$, intersects exactly twice $I\times\lbrace y\rbrace$, except one which is tangent to the slice.
\end{lemma}

\begin{proof}
The geodesic $\G_{\xi}(x_1+\mathbf{i}y)$ intersects twice the slice $[-3/2;3/2]\times\{y\}$. This implies in particular that there exists $x_0\in [-3/2;3/2]$ such that $\G_{\xi}(x_0+\mathbf{i}y)$ is tangent to $\lbrace z;\im(z)=y\rbrace$. If one prefers, $y=R_{x_0}=\xi-x_0$ (recall that $\xi>0$).

Now recall that when we apply the translation $\tau_x$ to the geodesics $\G_{\xi}(x+\mathbf{i}y)$ we get concentric circles. Thus we have for every $x,x'\in[-1/2;1/2]$, $|\re(N_x)-\re(N_{x'})|=|R_x-R_{x'}|\leq\Sup_{x\in[-1/2;1/2]}\frac{dR_x}{dx}.$ But we have already computed this derivative, and since $y=\xi-x_0\geq y_0$ for some $x_0\in [-3/2;3/2]$, we see that when $y_0$ is large enough, this derivative is uniformly bounded from above independently of $y\geq y_0$ and $\xi$.

This implies that all the North poles $N(x)$ have their real parts in an interval of uniformly bounded length: the geodesics $\G_{\xi}(x+\mathbf{i}y)$ intersect twice an interval which is at most twice bigger. We can conclude the proof of the lemma.
\end{proof}

This lemma proves that up to replacing the interval $[-1/2;1/2]$ by some interval $I$ \emph{of uniform length}, we are reduced to Hypothesis $(\ast\ast')$, which we just treated. This ends the proof of Proposition \ref{horizontal}. \quad \hfill $\square$

\paragraph{End of the proof of the integrability.} We will now show how to finish the proof of the integrability using Proposition \ref{horizontal}.

We will bound individually each of the three integrals defined above. Firstly, using a bound $\log(\xi^2+y^2)\leq C'\log(y)$ when $\xi\in[-1/2;1/2]$ we find:

$$I_{[-1/2;1/2]}\leq C\int_{-\frac{1}{2}}^{\frac{1}{2}}\left(\int_{y_0}^{\infty}\frac{\log(\xi^2+y^2)}{y^2}\,dy\right)d\xi\leq CC'\int_{y_0}^{\infty}\frac{\log\,y}{y^2}dy<\infty.$$

Secondly, by Proposition \ref{horizontal}:

\begin{eqnarray*}
I_{[1/2;\infty)}&\leq& C\int_{1/2}^{\infty}\int_{y_0}^{\infty}\frac{1}{y}\frac{\log(\xi^2+y^2)}{\xi^2+y^2}\,dy\,d\xi\\
                  &=&   C\int_{1/2}^{\infty}\int_{y_0}^{\infty}\frac{1}{y^3}\frac{\log((\xi/y)^2+1)+2\log y}{(\xi/y)^2+1}\,dy\,d\xi.
\end{eqnarray*}

Using a change of variable $u=\xi/y$ as well as the integrability on $[0;\infty)$ of the functions $u\mapsto\log(u^2+1)/(u^2+1)$ and $u\mapsto 1/(u^2+1)$ we obtain a constant $C'>0$ such that for every $y\geq y_0$:
$$\int_{1/2}^{\infty} \frac{\log((\xi/y)^2+1)}{(\xi/y)^2+1}\,d\xi,\,\int_{1/2}^{\infty} \frac{1}{(\xi/y)^2+1}\,d\xi\leq C'y.$$

Thus we find:
$$I_{[1/2;\infty)}\leq CC' \int_{y_0}^{\infty}\frac{1+2\log y}{y^2}dy<\infty.$$

Finally, a similar argument allows us to show that $I_{(-\infty,-1/2]}$ is finite, and this concludes the proof of Proposition \ref{integrability}, in the very particular case when $\DD$ is assumed to be the inclusion. The next paragraph shows how to deduce the general case from the study of the simple model.

\subsection{Integrability over the cusps: the general case}

\paragraph{Using the parabolicity.} The structure $(\DD,\rho)$ is parabolic: after Möbius changes of coordinates at the source and at the goal, it is possible to assume that $\DD:\Hyp_{\geq 1}\to\Hyp_{\geq 1}$ is a biholomorphism on its image which commutes with $z\mapsto z+1$.

It comes from the proof of Proposition \ref{sectionlipschitz}, that the modulus of $\DD'$ for the hyperbolic metric has to be bounded away from $0$ and $\infty$ in a fundamental domain of $z\mapsto z+1$ (which we may choose to be $C^+$): since $\DD$ commutes with the hyperbolic isometry $z\mapsto z+1$, it is bounded in the whole $\Hyp_{\geq 1}$. This implies that $\DD$ is bilipschitz in the whole $\Hyp_{\geq 1}$.

\paragraph{Controlled distortion in a box.} The key idea is that in the box $A_{\xi}(y)$ defined by \eqref{axi}, it is possible to control the distortion of spherical distance induced by $\DD$.

\begin{proposition}
\label{distortiondeveloping}
There exist constants $C>0$ and $\alpha>1$ such that for every $y\geq 1, \xi\in\R$ and $z_1,z_2\in A_{\xi}(y)$,
$$\dist_{\C\PP^1}(\DD(z_1),\DD(z_2))\geq\frac{C}{y^{\alpha}}\dist_{\C\PP^1}(z_1,z_2).$$
\end{proposition}

The fist step in the proof of this proposition is to prove:

\begin{lemma}
\label{localizeD}
There exist constants $C_1>0$, $\alpha>1$ such that for every $y\geq 1$, and $\xi\in\R$, $\DD(A_{\xi}(y))$ is included in the part of the complex plane identified with $[-C_1 y^{\alpha};C_1 y^{\alpha}]\times[C_1^{-1}y^{\alpha^{-1}};C_1 y^{\alpha}]$
\end{lemma}

\begin{proof} Consider $x\in\R$, $y\geq 1$, $z_0=x+\mathbf{i}$ and $z=x+\mathbf{i}y$. We have $\dist_{\Hyp}(z_0,z)=\log y$. Since $\DD$ is bilipschitz, the quantity $\Delta=\dist_{\Hyp}(\DD(z_0),\DD(z))$ lies in $[\alpha^{-1}\log y;\alpha\log y]$ for some $\alpha>1$.

A classical argument of hyperbolic geometry (see Section 5.9 of Thurston's notes \cite{T}) shows that since $\DD$ is bilipschitz in $\Hyp_{\geq 1}$ and fixes $\infty$, it sends the vertical geodesic ray $[z_0;\infty)$ onto a curve which stays at bounded distance, say $\delta$, of the vertical geodesic ray $[\DD(z_0);\infty)$. Moreover, $\DD$ is bounded in the compact set $[-1/2;1/2]\times\{1\}$, $\im\circ\DD$ is invariant by $z\mapsto z+1$, and $\re\circ\DD$ commute with this translation. Hence, $\re(\DD(z_0))-x$ and $\im(\DD(z_0))$ are uniformly bounded functions of $x$.

Define respectively $p_h$ and $p_e$ the orthogonal projections of $\DD(z)$ on $[\DD(z_0);\infty)$ for the hyperbolic and euclidian metric. Note that $p_e$ and $\DD(z)$ have the same imaginary part. By triangular inequality, the difference between $\Delta$ and $\dist_{\Hyp}(\DD(z_0),p_h)$ is uniformly bounded. Finally, since $\im(\DD(z_0))$ is uniformly bounded, the same holds for $\Delta-\log(\im(p_h))$.

Now, the curve $\DD([z_0;\infty))$ stays in a cone around the (complete) vertical geodesic passing through $\DD(z_0)$ whose angle is bounded by a quantity depending only on $\delta$. Elementary trigonometry implies that the ratio $\im(p_h)/\im(p_e)$ is uniformly log-bounded. We deduce that the difference $\Delta-\log\im(\DD(z))$ is uniformly bounded. A similar argument also shows that the difference $\re(\DD(z))-\re(\DD(z_0))$ stays in a log bounded ratio with $\im(\DD(z))$.

From all this, we deduce the existence of uniform $C'>0$ and $\alpha>1$ such that for every $x\in\R$ and $y\geq 1$, $\DD(x+\mathbf{i}y)$ lies in the part of the complex plane identified with  $[x-C'y^{\alpha};x+C' y^{\alpha}]\times[C'^{-1}y^{\alpha^{-1}};C' y^{\alpha}]$.

Now in order to finish the proof, use Lemma \ref{square}: when $x\in[-1/2;1/2]$, $A_{\xi}(y)$ stays in a box which is uniformly of the size of $y$.
\end{proof}

We now finish the proof of Proposition \ref{distortiondeveloping} using Lemma \ref{localizeD}. The proof of Lemma \ref{conformaldistortion} adapts to prove the existence of constants $C_2,C_3>0$ such that for every $y\geq 1$, and $z_1,z_2\in A_{\xi}(y)$, one has:
$$\dist_{\C\PP^1}(\DD(z_1),\DD(z_2))\geq \frac{C_2}{y^{\alpha}}\dist_{\Hyp}(\DD(z_1),\DD(z_2))\geq\frac{C_2 C_3}{y^{\alpha}}\dist_{\Hyp}(z_1,z_2)\geq \frac{C_2C_3}{y^{\alpha}}\dist_{\C\PP^1}(z_1,z_2),$$
where the last inequality is true because we recall that the hyperbolic and Fubini-Study metrics are conformally equivalent with a conformal factor $\leq 1$ in $\Hyp_{\geq 1}$.\quad \hfill $\square$

\paragraph{End of the story.} Recall that we want to prove that the function $\widetilde{\psi}$ defined by \eqref{lafonctionpsi} is Liouville-integrable in $C^+\times\R$. Since for every $x\in[-1/2;1/2]$, $y\geq 1$, $z=x+\mathbf{i}y$ and $\xi\in\R$, $G_{[0;1]}(z,\xi)\dans A_{\xi}(y)$, there is something to prove only when $\widetilde{s}^+(\xi)\in\DD(A_{\xi}(y))$. In that case, by Proposition \ref{distortiondeveloping}, we have for every $t\in[0;1]$,
$$\log(\dist_{\C\PP^1}(\widetilde{s}^+(\xi),\DD(G_t(z,\xi))))\leq C'+\alpha\log y+\log(\dist_{\C\PP^1}(\DD^{-1}(\widetilde{s}^+(\xi)),G_t(z,\xi))),$$
$C'$ being a uniform constant. Using the Liouville-integrability of $(x,y,\xi)\mapsto C'+\log y$ in $C^+\times\R$, we see that the Liouville-integrability of $\widetilde{\psi}$ is implied by that of a function that we already treated in the last paragraph. This allows us to conclude the proof of the general case of Proposition \ref{integrability}, and consequently that of Theorem \ref{limitedprojective}.\quad \hfill $\square$

\vspace{10pt}
\paragraph{Acknowledgments.} This paper grew out of several visits of the second author in the Institut de Mathématiques de Bourgogne, whose warm hospitality we wish to thank, which were financed by the ANR \emph{DynNonHyp} BLAN08-2 313375. It was terminated as both authors were Post-Doctorates in Rio de Janeiro, financed by CAPES-Brazil, respectively at IMPA and at the UFRJ. It benefited from decisive comments and ideas given by Christian Bonatti, Bertrand Deroin, Patrick Gabriel, Pablo Lessa, Graham Smith and Johan Taflin: we are very grateful to all of them. Last but not least, we wish to thank the anonymous reviewer, whose comments led us to precise the notion of parabolic projective structure.

\noindent \textbf{Sébastien Alvarez (sebastien.alvarez@u-bourgogne.fr)}\\
\noindent  Instituto de Matem{\'a}tica Pura e Aplicada.\\
\noindent  Est. D. Castorina 110, 22460-320, Rio de Janeiro, Brazil\\
\noindent \textbf{Nicolas Hussenot (nicolashussenot@hotmail.fr)}\\
\noindent  Instituto de Matem\'	atica, Universidade Federal do Rio de Janeiro\\
\noindent  P. O. Box, 68530, 21945-970 Rio de Janeiro, Brazil\\

\end{document}